\documentclass[12pt]{amsart}
\makeatletter

\@addtoreset{equation}{section}
\makeatother
\usepackage[margin=3cm]{geometry}
\newenvironment{nouppercase}{%
  \renewcommand{\uppercasenonmath}[1]{}}{}
\usepackage{amsmath,amssymb,amsthm,amsfonts}
\usepackage[T1]{fontenc}
\usepackage{ytableau}
\ytableausetup{boxsize=normal,aligntableaux=center}
\allowdisplaybreaks[1]
\theoremstyle{definition}
\newtheorem{theorem}[equation]{Theorem}
\newtheorem{prop}[equation]{Proposition} 
\newtheorem{definition}[equation]{Definition} 
\newtheorem{remark}[equation]{Remark} 
\newtheorem{cor}[equation]{Corollary}
\newtheorem{prob}[equation]{Problem}
\newtheorem{lemma}[equation]{Lemma}
\newtheorem{exam}[equation]{Example}
\newtheorem{introthm}{Theorem}

\def\zetas{\zeta^{\star}}
\begin{document}

\title{Jack multiple zeta function}
\author{Maki Nakasuji}
\address{Department of Information and Communication Science, Faculty of Science, Sophia
University, 7-1 Kioi-cho, Chiyoda-ku, Tokyo, 102-8554, Japan}
\email{ nakasuji@sophia.ac.jp}
\author[W. Takeda]{Wataru Takeda}
\address{Department of Mathematics, Toho University,
2-2-1, Miyama, Funabashi-shi, Chiba 274-8510, Japan.}
\email{wataru.takeda@sci.toho-u.ac.jp}

\begin{nouppercase}
\maketitle
\end{nouppercase}
\begin{abstract}
For conjugate shapes $\lambda$ and $\lambda'$, Stanley's duality relates the Jack symmetric functions $J_\lambda(x;\alpha)$ and $J_{\lambda'}(x;1/\alpha)$. Through Stanley's combinatorial formula, it is an identity between the sums of the weights of the semi-standard Young tableaux giving the same monomial.
We introduce the Jack multiple zeta functions, which keep these weights separate, and prove that Stanley's duality lifts to them, including the case of skew shapes.
Thus the duality holds before the weights are summed up, and in fact step by step along chains of partitions. Since the duality is an identity among formal symbols, the same duality holds when the multiple zeta functions are replaced by other functions, and replacing them with the monomial quasi-symmetric functions recovers Stanley's duality.
Hence our duality is a refinement of Stanley's duality which contains it as a special case. In particular, we obtain a duality for the Schur multiple zeta functions and the classical relation between the multiple zeta and zeta-star functions.
We next show that any expansion of a Jack symmetric function in terms of Schur functions can be lifted to Jack multiple zeta functions, yielding an analog of a formula conjectured by Stanley and proved by Koike. 
We also show that the Jack multiple zeta functions of one-row shape give a one-parameter family interpolating the multiple zeta and zeta-star functions, and we discuss their values at $\{2\}^r$ and at $\{2\}^a,3,\{2\}^b$.
\end{abstract}
\section{Introduction}\label{sec:intro}

\subsection{Stanley's duality and the question of this paper}

The Jack symmetric functions $J_\lambda(x;\alpha)$ introduced by Jack \cite{ja70} are a family of symmetric functions that depend on a parameter $\alpha$. More precisely, let $\langle\cdot,\cdot\rangle_\alpha$ be the inner product on the ring of symmetric functions defined by $\langle p_\lambda,p_\mu\rangle_\alpha=\delta_{\lambda\mu}z_\lambda\alpha^{\ell(\lambda)}$ where $p_\lambda$ denotes the power sum symmetric function, $\delta_{\lambda\mu}$ is the Kronecker delta, $z_\lambda=\prod_{i\ge1}i^{m_i(\lambda)}m_i(\lambda)!$, and $\ell(\lambda)$ is the number of parts of $\lambda$. 
Here $m_i(\lambda)$ denotes the multiplicity of $i$ in $\lambda$. Then the $J_\lambda(x;\alpha)$ form the orthogonal basis which is triangular with respect to the monomial symmetric functions in the dominance order. They are normalized so that the coefficient of $x_1x_2\cdots x_n$ is $n!$ for a partition $\lambda$ of $n$.
This is the integral normalization of Stanley \cite{st89}. When $\alpha=1$, they are the Schur functions up to a constant, that is, $J_\lambda(x;1)=h_\lambda s_\lambda(x)$, where $h_\lambda$ is the product of the hook lengths of $\lambda$. When $\alpha=2$, they are the zonal polynomials, which are the zonal spherical functions of the symmetric pair $(GL(n,\mathbb{R}),O(n))$ (see \cite[Chapters~VI and VII]{mac}). 

Throughout this paper, we identify a partition $\lambda=(\lambda_1,\lambda_2,\ldots)$ with its Young diagram, that is, the set of boxes $(i,j)$ with $1\le j\le\lambda_i$. 
Thus $(i,j)\in\lambda$ means that $(i,j)$ is a box of $\lambda$, $\mu\subset\lambda$ means the inclusion of the Young diagrams, and $\lambda/\mu$ is the set of boxes of $\lambda$ not in $\mu$.
We call a box $(i,j)\in\lambda$ a corner of $\lambda$ if neither $(i+1,j)$ nor $(i,j+1)$ belongs to $\lambda$, and we write $\lambda'$ for the conjugate partition. For partitions $\mu\subset\lambda$, we denote by ${\rm SSYT}(\lambda/\mu)$ the set of semi-standard Young tableaux $M=(m_{ij})$ of skew shape $\lambda/\mu$ with positive integer entries,
whose entries are weakly increasing from left to right in each row and strictly increasing from top to bottom in each column.
We write ${\rm SSYT}(\lambda)={\rm SSYT}(\lambda/\emptyset)$.

Stanley \cite{st89} studied the combinatorial properties of the Jack symmetric functions in detail. Two of his results are the starting point of this paper. The first one is a combinatorial formula
\begin{equation}\label{eqn:intro_stanley}
J_{\lambda/\mu}(x;\alpha)=\sum_{M\in{\rm SSYT}(\lambda/\mu)}w(M)x^M,\qquad x^M=\prod_{(i,j)\in\lambda/\mu}x_{m_{ij}},
\end{equation}
where $w(M)$ is an explicit rational function of $\alpha$ \cite[Theorem~6.3]{st89} (see Theorem \ref{thm:stanley63}). The second one is a duality between the shapes $\lambda$ and $\lambda'$ \cite[Theorem~3.3]{st89}:
\begin{equation}\label{eqn:intro_stanleydual}
(-1)^{|\lambda|}\omega_{-\frac1\alpha}\bigl(J_{\lambda'}(x;1/\alpha)\bigr)=\alpha^{-|\lambda|}J_\lambda(x;\alpha),
\end{equation}
where $\omega_\beta$ is the algebra endomorphism of the ring of symmetric functions with $\omega_\beta(p_r)=\beta p_r$ for $r\ge1$. In \eqref{eqn:intro_stanley}, the monomial $x^M$ depends only on how many times each integer occurs in $M$, so that $J_\lambda$ depends on the weights $w(M)$ only through their sums over the tableaux giving the same monomial.
Accordingly, \eqref{eqn:intro_stanleydual} is an identity between these sums. The question of this paper is whether the duality \eqref{eqn:intro_stanleydual} already holds before these sums are taken, that is, at the level of the individual weights of Stanley's formula.
We show that this is the case, by replacing the monomials $x^M$ with multiple zeta functions.

\subsection{Schur multiple zeta functions}

The multiple zeta functions and the multiple zeta-star functions of the Euler--Zagier type (see \cite{h97,z94}) are defined by
\[\zeta(s_1,\ldots,s_r)=\sum_{0<m_1<\cdots<m_r}\frac{1}{m_1^{s_1}\cdots m_r^{s_r}},\qquad \zetas(s_1,\ldots,s_r)=\sum_{0<m_1\le\cdots\le m_r}\frac{1}{m_1^{s_1}\cdots m_r^{s_r}}.\]
Nakasuji, Phuksuwan and Yamasaki \cite{npy} unified them by the Schur multiple zeta functions. For a given set $\pmb s=(s_{ij})$ of variables attached to the boxes $(i,j)$ of $\lambda$, the Schur multiple zeta function of shape $\lambda$ is defined to be
\[\zeta_\lambda(\pmb s)=\sum_{M\in{\rm SSYT}(\lambda)}\frac{1}{M^{\pmb s}},\qquad M^{\pmb s}=\prod_{(i,j)\in\lambda}m_{ij}^{s_{ij}},\]
which absolutely converges if $\Re(s_{ij})\ge1$ for all $(i,j)$ and $\Re(s_{ij})>1$ for all corners $(i,j)$ of $\lambda$ \cite{npy}. We have
\begin{equation}\label{eqn:intro_zeta}
\ytableausetup{boxsize=normal}
\zeta_{(1^r)}\left(\ytableaushort{{s_1},{\vdots},{s_r}}\right)=\zeta(s_1,\ldots,s_r),\qquad \zeta_{(r)}\left(\ytableaushort{{s_1}{\cdots}{s_r}}\right)=\zetas(s_1,\ldots,s_r).
\end{equation}
Since the Schur multiple zeta functions have the structures of both the Schur functions and the multiple zeta functions, they are expected to have properties similar to those of both.
Indeed, as for the Schur functions, the Jacobi--Trudi and Giambelli formulas were shown in \cite{npy}. After that, the Pieri formulas were obtained for hook type \cite{nt} and in general \cite{nak}, and the Littlewood--Richardson rule in \cite{nak,ha1}.
As for the multiple zeta functions, the integral expressions for ribbon type were given in \cite{npy}, the shuffle product formula for hook type in \cite{ntsh}, and the duality formula and its generalization in \cite{no}. The last one was interpolated to complex functions in \cite{not}.
Moreover, the multiple zeta functions with structures similar to those of other symmetric functions such as Schur $P$-, Schur $Q$-, symplectic and orthogonal functions were introduced in \cite{nt25}; see also \cite{ha2}.

\subsection{Jack multiple zeta functions and the main result}

Following \eqref{eqn:intro_stanley}, we define the Jack multiple zeta function of shape $\lambda/\mu$ to be
\[\zeta^J_{\lambda/\mu}(\pmb s;\alpha)=\sum_{M\in{\rm SSYT}(\lambda/\mu)}\frac{w(M)}{M^{\pmb s}},\]
where $\pmb s=(s_{ij})$ is a set of variables attached to the boxes of $\lambda/\mu$, and we write $\zeta^J_\lambda=\zeta^J_{\lambda/\emptyset}$. When $\alpha=1$, all the weights $w(M)$ for $M\in{\rm SSYT}(\lambda)$ are equal to $h_\lambda$, and hence $\zeta^J_{\lambda}(\pmb s;1)=h_{\lambda}\zeta_{\lambda}(\pmb s)$.
Also, for $\lambda=(1^r)$, we have $\zeta^J_{(1^r)}(\pmb s;\alpha)=r!\,\zeta(s_1,\ldots,s_r)$ for any $\alpha$. Unlike $x^M$, the term $M^{-\pmb s}$ remembers which boxes of $M$ carry equal entries. Therefore, in $\zeta^J_{\lambda/\mu}$, tableaux giving the same monomial in \eqref{eqn:intro_stanley} may contribute to different multiple zeta functions, and $\zeta^J_{\lambda/\mu}$ keeps the weights which $J_{\lambda/\mu}$ adds up.

As an analog of $\omega_\beta$, for a parameter $\beta$ we consider the linear map $\omega_\beta$ on multiple zeta functions which merges consecutive entries with binomial coefficients:
\[\omega_\beta\bigl(\zeta(t_1,\ldots,t_\ell)\bigr)
=\sum_{\substack{r_1+\cdots+r_k=\ell\\ r_i\ge1}}\ \prod_{i=1}^{k}\binom{\beta}{r_i}\ \zeta\bigl(t_1+\cdots+t_{r_1},\ \ldots,\ t_{\ell-r_k+1}+\cdots+t_\ell\bigr).
\]
These maps satisfy $\omega_\beta\circ\omega_\gamma=\omega_{\beta\gamma}$. For variables $\pmb s=(s_{ij})$ on $\lambda/\mu$, we write $\pmb s'$ for the transposed variables on $\lambda'/\mu'$, whose box $(i,j)$ carries $s_{ji}$. Then our main result is the following.

\begin{introthm}[Theorem \ref{thm:general}]\label{thm:intro3}
For any partitions $\mu\subset\lambda$, we have
\[(-1)^{|\lambda|-|\mu|}\omega_{-\frac{1}{\alpha}}\left(\zeta_{\lambda'/\mu'}^J\left(\pmb s';\frac{1}{\alpha}\right)\right)
=\frac{1}{\alpha^{|\lambda|+|\mu|}}\zeta_{\lambda/\mu}^J\left(\pmb s;\alpha\right).\]
In particular, for any partition $\lambda$ of $n$, we have
\[(-1)^{n}\omega_{-\frac{1}{\alpha}}\left(\zeta_{\lambda'}^J\left(\pmb s';\frac{1}{\alpha}\right)\right)=\frac{1}{\alpha^{n}}\zeta_{\lambda}^J\left(\pmb s;\alpha\right).\]

\end{introthm}

When $\alpha=1$, Theorem \ref{thm:intro3} gives the duality $(-1)^n\omega_{-1}(\zeta_{\lambda'}(\pmb s'))=\zeta_\lambda(\pmb s)$ for the Schur multiple zeta functions, and for $\lambda=(r)$ this is exactly the classical expression of $\zetas(s_1,\ldots,s_r)$ as the sum of the multiple zeta functions obtained by merging consecutive variables.

Theorem \ref{thm:intro3} is finer than \eqref{eqn:intro_stanleydual}. For example, let $\lambda=(2,1)$, which is self-conjugate. For $i<j<k$, the monomial $x_ix_jx_k$ in $J_{(2,1)}$ comes from two tableaux, according to whether $j$ is in the box $(1,2)$ or in the box $(2,1)$, with the weights
\[\frac{2(\alpha+2)}{\alpha+1}\qquad\text{and}\qquad\frac{2(2\alpha+1)}{\alpha+1},\]
respectively, and \eqref{eqn:intro_stanleydual} only involves their sum $6$. In $\zeta^J_{(2,1)}(\pmb s;\alpha)$, these two weights are the coefficients of the two different terms $\zeta(s_{11},s_{12},s_{21})$ and $\zeta(s_{11},s_{21},s_{12})$.
On the left-hand side of Theorem \ref{thm:intro3}, the transposition of the variables interchanges these two terms, and replacing $\alpha$ by $1/\alpha$ interchanges the two weights. Therefore, Theorem \ref{thm:intro3} matches the two tableaux individually. The other terms, on which $\omega_{-1/\alpha}$ acts nontrivially,
are matched in the same way (see Example \ref{ex:factor21}).

The proof of Theorem \ref{thm:intro3} shows that the duality holds step by step. Both sides of Theorem \ref{thm:intro3} are linear combinations of the multiple zeta functions $\zeta(\pmb u)$ attached to sequences of partitions $\pmb u:\mu=\lambda^{(0)}\subsetneq\lambda^{(1)}\subsetneq\cdots\subsetneq\lambda^{(m)}=\lambda$.
On the right-hand side, the coefficient of $\zeta(\pmb u)$ is a product over the steps $\lambda^{(k-1)}\subset\lambda^{(k)}$ by Stanley's formula, and we show that the same holds on the left-hand side (Lemma \ref{lem:factor}). 
Thus Theorem \ref{thm:intro3} reduces to an identity for each single step, and this identity follows from the skew version of \eqref{eqn:intro_stanleydual} (Proposition \ref{prop:skewdual}) by specializing the variables to $(1^z)$ and putting $z=1$ (Lemma \ref{lem:local}).
The point is that the binomial coefficients $\binom{-1/\alpha}{r}$ in $\omega_{-1/\alpha}$ are exactly what the specialization of $\omega_{-1/\alpha}$ on symmetric functions produces, namely $(\omega_\beta f)(1^z)=f(1^{\beta z})$ (Lemma \ref{lem:omspec}).
In this sense, Stanley's duality is used only through its skew version at a single specialization.

Since the proof uses only the formal independence of the multiple zeta functions, Theorem \ref{thm:intro3} also holds for the functions obtained by combining Stanley's weights with the generalization of the Schur functions by Bachmann and Charlton \cite{bc}. In particular, replacing the multiple zeta functions with the monomial quasi-symmetric functions, we recover the skew version of Stanley's duality \eqref{eqn:intro_stanleydual}.
Thus Theorem \ref{thm:intro3} is a refinement of Stanley's duality which contains it as a special case (Remark \ref{rem:general_f}).

\subsection{Further results}

We next consider the expansion of $J_\lambda$ in terms of the Schur functions. For general $\alpha$, the weight $w(M)$ depends on the positions of the equal entries of $M$, and such an expansion cannot be lifted directly to $\zeta^J_\lambda$.
To overcome this problem, we take the summation of the Jack multiple zeta functions over the permutations of the variables, as in \cite{nt}. Here, $\sum_{\rm sym}$ means the summation over the permutations of all the variables except for the one in the box $(1,1)$. 
\begin{introthm}[Theorem \ref{thm:sym_transfer}]\label{thm:intro2}
Let $\lambda$ be a partition of $n$ and let
\[J_\lambda(x;\alpha)=\sum_{\mu}c_{\lambda\mu}(\alpha)s_\mu(x)\]
be the expansion of $J_\lambda$ in the Schur basis. Then we have
\[\sum_{\rm sym}\zeta^J_\lambda(\pmb s;\alpha)=\sum_{\mu}c_{\lambda\mu}(\alpha)\sum_{\rm sym}\zeta_\mu(\pmb s).\]
\end{introthm}

In particular, applying Theorem \ref{thm:intro2} to a formula conjectured by Stanley and subsequently proved by Koike \cite{ko93} (Theorem \ref{thm:stanley47}), we obtain an explicit formula for $\lambda=(2^i,1^j)$.

Finally, we consider the one-row case. By \eqref{eqn:intro_zeta}, the functions $\zeta^J_{(r)}$ are the natural candidates to connect the multiple zeta and zeta-star functions. 
\begin{introthm}[Lemma \ref{lem:1r}, Corollaries \ref{cor:interp} and \ref{cor:s2}]\label{thm:intro1}
For an integer $t\ge1$, put $g_\alpha(t)=\frac{1}{t!}\prod_{j=1}^{t-1}(1+\alpha j)$. Then for $r\ge1$ we have \[
\ytableausetup{boxsize=normal}
\frac{1}{r!}\zeta_{(r)}^J\left(\ytableaushort{{s_{1}}{\cdots}{s_{r}}};\alpha\right)
=\sum_{\substack{\boldsymbol{t}=(s_1 \square s_2 \square \cdots \square s_r)\\{\rm type}(\boldsymbol{t})=(t_1,\ldots,t_\ell)}}
\prod_{i=1}^{\ell}g_\alpha(t_i)\,\zeta(\boldsymbol{t}),
\]
where $\boldsymbol{t}$ runs over all indices obtained from $(s_1,\ldots,s_r)$ by replacing each $\square$ by a comma or a plus sign, and $t_i$ is the number of variables in the $i$-th entry of $\boldsymbol{t}$. In particular, $\frac{1}{r!}\zeta^J_{(r)}$ is equal to $\zeta(s_1,\ldots,s_r)$ when $\alpha=-1$ and to $\zetas(s_1,\ldots,s_r)$ when $\alpha=1$. Moreover, for $\alpha\ne0$ we have
\[
\sum_{r\ge0}\frac{1}{r!}\zeta_{(r)}^J\left(\{2\}^r;\alpha\right)x^r
=\left(\frac{\sin \pi\sqrt{\alpha x}}{\pi\sqrt{\alpha x}}\right)^{-\frac1\alpha}.
\]
\end{introthm}

Theorem \ref{thm:intro1} shows that the one-row Jack multiple zeta functions interpolate the multiple zeta and zeta-star functions, and this interpolation is different from that of Yamamoto \cite{y13}.
We also consider the case where one $3$ is inserted among $2$'s as in Zagier's evaluation \cite{z12}. We obtain a generating function and a reflection formula between the parameters $\alpha$ and $-\alpha$ (Theorem \ref{thm:refl}). At $\alpha=1$, it reduces to the relation between the multiple zeta-star values and the multiple zeta values of this type found by Zagier, so that it interpolates Zagier's relation by the parameter $\alpha$. It also gives a divisibility by $1-\alpha^2$ modulo Zagier's space (Corollary \ref{cor:refl}).
Numerical experiments in weight $11$ suggest further structure, which we formulate as a problem.

\subsection{Organization of the paper}

In Section \ref{sec:prelim}, we review Stanley's results on the Jack symmetric functions, and in Section \ref{sec:jackzeta}, we define the Jack multiple zeta functions and the maps $\omega_\beta$. In Section \ref{sec:duality}, we prove Theorem \ref{thm:intro3}. In Section \ref{sec:2i1j}, we prove Theorem \ref{thm:intro2} and give the formula for $\lambda=(2^i,1^j)$. In Section \ref{sec:omega}, we study the one-row case, where we prove Theorem \ref{thm:intro1} and discuss the values at $\{2\}^a,3,\{2\}^b$.
In Appendix \ref{app:r1}, we give a direct proof of Theorem \ref{thm:intro3} for $\lambda=(r,1)$ by a telescoping identity.

\section{Preliminaries}\label{sec:prelim}

In this section, we recall Stanley's combinatorial formula and duality for the Jack symmetric functions. Throughout this paper, $(x)_t$ denotes the falling factorial
\[(x)_t=x(x-1)\cdots(x-t+1)\qquad (t\ge1),\qquad (x)_0=1,\]
so that $\dfrac{(x)_t}{t!}=\dbinom{x}{t}$ and $\displaystyle\sum_{t\ge0}\binom{x}{t}y^t=(1+y)^x$.

For a partition $\lambda$, we write $\lambda'$ for the conjugate partition, so that $\lambda'_j$ is the length of the $j$-th column. For a box $s=(i,j)\in\lambda$ we write $a_\lambda(s)=\lambda_i-j$ and $l_\lambda(s)=\lambda'_j-i$ for its arm and leg. For partitions $\mu\subset\lambda$, the skew diagram $\lambda/\mu$ is a horizontal strip if it contains at most one box in each column.
We denote by $J_\lambda(x;\alpha)$ the Jack symmetric function with parameter $\alpha$ in the integral normalization of \cite{st89} (see also \cite[Chapter~VI, \S10]{mac}), and by $J_{\lambda/\mu}(x;\alpha)$ the skew Jack symmetric function of \cite[\S6]{st89}.

For partitions $\mu\subset\lambda$ we define $B_{\lambda,\mu}(s)$ for $s=(i,j)\in\lambda$ and $C_{\lambda,\mu}(s)$ for $s=(i,j)\in\mu$ by
\[B_{\lambda,\mu}((i,j))=\left\{\begin{array}{cl}
  \lambda_j'-i+\alpha(\lambda_i-j+1)   &  \text{if $\lambda/\mu$ does not contain a box}\\
  &\text{in the same column as $(i,j)$,}\\
  \lambda_j'-i+1+\alpha(\lambda_i-j)    & \text{otherwise,}
\end{array}
\right.\]
\[C_{\lambda,\mu}((i,j))=\left\{\begin{array}{cl}
  \mu_j'-i+\alpha(\mu_i-j+1)   &  \text{if $\lambda/\mu$ does not contain a box}\\
  &\text{in the same column as $(i,j)$,}\\
  \mu_j'-i+1+\alpha(\mu_i-j)    & \text{otherwise,}
\end{array}
\right.\]
and
\[j_\lambda=\prod_{s\in\lambda}\bigl(l_\lambda(s)+\alpha(a_\lambda(s)+1)\bigr)\bigl(l_\lambda(s)+1+\alpha a_\lambda(s)\bigr).\]

\begin{theorem}[{\cite[Theorem 6.3]{st89}}]\label{thm:stanley63}
    \[J_{\lambda/\mu}=\sum_{T}w(T)x^T,\]
    summed over all semi-standard Young tableaux $T$ of shape $\lambda/\mu$. Here $T$ is identified with the sequence $\mu=\lambda^{(0)}\subset\lambda^{(1)}\subset\cdots$, where $\lambda^{(i)}$ consists of $\mu$ and the boxes of $T$ with entries at most $i$.
Each $\lambda^{(i)}/\lambda^{(i-1)}$ is a (possibly empty) horizontal strip, and $\lambda^{(i)}=\lambda$ for $i$ large. We put $x^T=\prod_{i\ge1}x_i^{|\lambda^{(i)}/\lambda^{(i-1)}|}$ and, for any $r$ with $\lambda^{(r)}=\lambda$,
    \[w(T)=\frac{\displaystyle j_{\mu}\prod_{i=1}^r\prod_{s\in\lambda^{(i)}}B_{\lambda^{(i)},\lambda^{(i-1)}}(s)}{\displaystyle \prod_{i=1}^r\prod_{s\in\lambda^{(i-1)}}C_{\lambda^{(i)},\lambda^{(i-1)}}(s)}.\]
    This does not depend on $r$, because a step with $\lambda^{(i)}=\lambda^{(i-1)}$ contributes the factor $1$.
\end{theorem}
\begin{exam}\label{ex:intro21}
Discarding the empty steps, a tableau $T$ of shape $\lambda$ gives a chain of distinct shapes, and $w(T)$ is the product over its steps $\mu\subset\nu$ of the factors
\[\prod_{s\in\nu}B_{\nu,\mu}(s)\Big/\prod_{s\in\mu}C_{\nu,\mu}(s).\]
We list the values of $B$ and $C$ for the boxes in the order $(1,1),(1,2),(2,1)$:
\[
\begin{array}{l|l|l|c}
\text{step }\mu\subset\nu & B_{\nu,\mu}(s),\ s\in\nu & C_{\nu,\mu}(s),\ s\in\mu & \text{factor}\\\hline
\emptyset\subset(1) & 1 & & 1\\
\emptyset\subset(2) & 1+\alpha,\ 1 & & 1+\alpha\\
(1)\subset(2) & 2\alpha,\ 1 & \alpha & 2\\
(1)\subset(1,1) & 2,\ 1 & 1 & 2\\
(1)\subset(2,1) & \alpha+2,\ 1,\ 1 & 1 & \alpha+2\\[2pt]
(2)\subset(2,1) & \alpha+2,\ \alpha,\ 1 & \alpha+1,\ \alpha & \frac{\alpha+2}{\alpha+1}\\[4pt]
(1,1)\subset(2,1) & 2\alpha+1,\ 1,\ \alpha & \alpha+1,\ \alpha & \frac{2\alpha+1}{\alpha+1}
\end{array}
\]
\begin{enumerate}
\item Let $\lambda=(2)$. There are two chains:
\[
\begin{array}{l|l|l|c}
\text{chain} & \text{boxes added} & \text{monomials} & \text{weight}\\\hline
\emptyset\subset(2) & (1,1),(1,2) & x_i^2 & 1+\alpha\\
\emptyset\subset(1)\subset(2) & (1,1);\ (1,2) & x_ix_j\ \ (i<j) & 1\cdot2
\end{array}
\]
Therefore
\[J_{(2)}=(\alpha+1)\sum_{i}x_i^2+2\sum_{i<j}x_ix_j.\]
\item Let $\lambda=(2,1)$. There are four chains:
\[
\setlength{\arraycolsep}{3.5pt}
\begin{array}{l|l|l|c}
\text{chain} & \text{boxes added} & \text{monomials} & \text{weight}\\\hline
\emptyset\subset(1)\subset(2)\subset(2,1) & (1,1);\ (1,2);\ (2,1) & x_ix_jx_k\ \ (i<j<k) & 1\cdot2\cdot\frac{\alpha+2}{\alpha+1}\\[4pt]
\emptyset\subset(1)\subset(1,1)\subset(2,1) & (1,1);\ (2,1);\ (1,2) & x_ix_jx_k\ \ (i<j<k) & 1\cdot2\cdot\frac{2\alpha+1}{\alpha+1}\\[4pt]
\emptyset\subset(1)\subset(2,1) & (1,1);\ (1,2),(2,1) & x_ix_j^2\ \ (i<j) & 1\cdot(\alpha+2)\\[2pt]
\emptyset\subset(2)\subset(2,1) & (1,1),(1,2);\ (2,1) & x_i^2x_j\ \ (i<j) & (1+\alpha)\cdot\frac{\alpha+2}{\alpha+1}
\end{array}
\]
For instance, in the first chain the boxes $(1,1)$, $(1,2)$ and $(2,1)$ carry $x_i$, $x_j$ and $x_k$, respectively. Since $\frac{2(\alpha+2)}{\alpha+1}+\frac{2(2\alpha+1)}{\alpha+1}=6$, we obtain
\begin{align*}
J_{(2,1)}&=6\sum_{i<j<k}x_ix_jx_k+(\alpha+2)\sum_{i<j}x_ix_j^2+(\alpha+2)\sum_{i<j}x_i^2x_j\\
&=6\sum_{i<j<k}x_ix_jx_k+(\alpha+2)\sum_{i\ne j}x_i^2x_j.
\end{align*}
\end{enumerate}
\end{exam}

Finally, we recall the duality for Jack symmetric functions due to Stanley \cite[Theorem~3.3]{st89}. Let $\Lambda$ be the ring of symmetric functions over $\mathbb{Q}(\alpha)$. For $r\ge1$ let $p_r=\sum_i x_i^r$ be the $r$-th power sum symmetric function. Define an algebra automorphism
\[
\omega_\beta:\Lambda\to\Lambda
\]
by its action on power sums:
\[
\omega_\beta(p_r)=\beta p_r.
\]

\begin{theorem}[Stanley {\cite[Theorem~3.3]{st89}}]\label{thm:stanley33}
For every partition $\lambda$ we have
\[
(-1)^{|\lambda|}\omega_{-\frac{1}{\alpha}}(J_{\lambda'}(x;1/\alpha))=\alpha^{-|\lambda|}J_\lambda(x;\alpha).
\]
\end{theorem}

\section{Jack multiple zeta functions and the maps $\omega_\beta$}\label{sec:jackzeta}

In this section, we define the Jack multiple zeta functions and the maps $\omega_\beta$ on them. 

\begin{definition}[Jack multiple zeta function]\label{def:jackzeta}
Let $\mu\subset\lambda$ be partitions and let $\pmb s=(s_{ij})_{(i,j)\in\lambda/\mu}$ be variables attached to the boxes of $\lambda/\mu$. We define
\[\zeta^{J}_{\lambda/\mu}(\pmb s;\alpha)=\sum_{M}\frac{w(M)}{M^{\pmb s}},\qquad
M^{\pmb s}=\prod_{(i,j)\in\lambda/\mu}m_{ij}^{s_{ij}},\]
summed over all semi-standard Young tableaux $M=(m_{ij})$ of shape $\lambda/\mu$. Here $w(M)$ is the weight $w(T)$ of Theorem \ref{thm:stanley63}. When $\mu=\emptyset$, we write $\zeta^{J}_{\lambda}(\pmb s;\alpha)$ for $\zeta^{J}_{\lambda/\emptyset}(\pmb s;\alpha)$.
\end{definition}
\begin{exam}
\ytableausetup{boxsize=normal}
Let $\lambda=(2)$. By Example \ref{ex:intro21}(1), the chain $\emptyset\subset(1)\subset(2)$, in which the boxes $(1,1)$ and $(1,2)$ are added at different steps, has the weight $2$ and gives $\zeta(s_{11},s_{12})$, and the chain $\emptyset\subset(2)$ has the weight $1+\alpha$ and gives $\zeta(s_{11}+s_{12})$. Therefore,
\begin{align*}
\zeta^{J}_{(2)}\left(\ytableaushort{{s_{11}}{s_{12}}};\alpha\right)
&=2\zeta(s_{11},s_{12})+(1+\alpha)\zeta(s_{11}+s_{12})\\
&=(1-\alpha)\zeta_{(1,1)}\left(\ytableaushort{{s_{11}},{s_{12}}}\right)+(1+\alpha)\zeta_{(2)}\left(\ytableaushort{{s_{11}}{s_{12}}}\right).
\end{align*}
\end{exam}

Throughout this paper, we regard multiple zeta functions as formal symbols. To make this precise, let $S=\{s_1,s_2,\ldots\}$ be a set of independent variables and let
\[
\mathcal{N}=\left\{\left.\sum_{i\ge1}c_is_i\ \right|\ c_i\in\mathbb Z_{\ge0},\ \sum_i c_i\ge1,\ c_i=0 \text{ for all but finitely many } i\right\}
\]
be the set of nonzero $\mathbb Z_{\ge0}$-linear combinations of the variables, so that every index occurring in this paper is a tuple of elements of $\mathcal{N}$. We denote by
\[
\mathfrak Z=\bigoplus_{\ell\ge1}\ \bigoplus_{\boldsymbol{t}\in \mathcal{N}^\ell}\mathbb Q(\alpha)\,\zeta(\boldsymbol{t})
\]
the $\mathbb Q(\alpha)$-vector space freely spanned by the symbols $\zeta(\boldsymbol{t})$, one for each finite tuple $\boldsymbol{t}=(t_1,\ldots,t_\ell)$ of elements of $\mathcal{N}$. Any identity in $\mathfrak Z$ may then be specialized to actual multiple zeta functions by substituting complex numbers for the variables in the region of convergence. Grouping the tableaux $M$ in Definition \ref{def:jackzeta} according to the positions of their equal entries,
we write $\zeta^J_{\lambda/\mu}(\pmb s;\alpha)$ as a $\mathbb Q(\alpha)$-linear combination of the $\zeta(\boldsymbol{t})$, and we regard it as an element of $\mathfrak Z$.

We now introduce a zeta analog of the map $\omega_\beta$ on $\Lambda$ in Section \ref{sec:prelim}. For elements $s_1,\ldots,s_r$ of $\mathcal{N}$, we put $\boldsymbol{s} =(s_1,\ldots,s_r)$ and consider the tuples
\[
\boldsymbol{t}=(s_1 \square s_2 \square \cdots \square s_r).
\]
For each $\square$, we replace it either by a plus sign $+$ or by a comma $,$. For this replacement of $\boldsymbol{s}$, we define its type as follows. The commas divide $s_1,\ldots,s_r$ into blocks of consecutive entries. We define
\begin{equation}
    \label{type}
    {\rm type}(\boldsymbol{t}) = (t_1,\dots,t_\ell),
\end{equation}
where $t_i$ is the number of the $s_j$ in the $i$-th block. This is a composition of $r$.

For ${\rm type}(\boldsymbol{t})=(t_1,t_2,\ldots,t_\ell)$, we define
\[\binom{\beta}{{\rm type}(\boldsymbol{t})}=\prod_{i=1}^\ell\binom{\beta}{t_i}=\prod_{i=1}^\ell\frac{(\beta)_{t_i}}{t_i!}.\]

Now we define a $\mathbb Q(\alpha)$-linear map on the space $\mathfrak Z$,
\[
\omega_\beta:\mathfrak Z\longrightarrow\mathfrak Z
\]
by its action on the basis:
\[\omega_\beta(\zeta(s_1,s_2,\ldots,s_r))=\sum_{\boldsymbol{t}=(s_1 \square s_2 \square \cdots \square s_r)}\binom{\beta}{{\rm type}(\boldsymbol{t})}\zeta(\boldsymbol{t}),\]
where the sum runs over all replacements of $\boldsymbol{s}$. We note that $\omega_1=\rm id$. By Corollary \ref{cor:omega_inverse} below, $\omega_\beta$ is an automorphism of $\mathfrak Z$ for every $\beta\ne0$, with inverse $\omega_{1/\beta}$.

\begin{exam}
    \begin{align*}
        \omega_\beta(\zeta(s_1,s_2,s_3))&=\sum_{\boldsymbol{t}=(s_1 \square s_2 \square s_3)}\binom{\beta}{{\rm type}(\boldsymbol{t})}\zeta(\boldsymbol{t})\\
        &=\binom{\beta}{1}^{3}\zeta(s_1,s_2,s_3)+\binom{\beta}{2}\binom{\beta}{1}\zeta(s_1+s_2,s_3)\\
        &\qquad+\binom{\beta}{1}\binom{\beta}{2}\zeta(s_1,s_2+s_3)+\binom{\beta}{3}\zeta(s_1+s_2+s_3).
    \end{align*}
\end{exam}

\begin{theorem}\label{thm:omega_group}
For any $\beta,\gamma$, we have
\[\omega_\beta\circ\omega_{\gamma}=\omega_{\beta\gamma}.\]
\end{theorem}

\begin{proof}
Let $\boldsymbol{s}=(s_1,\ldots,s_n)$. By definition, $\omega_\beta\left(\omega_\gamma(\zeta(\boldsymbol{s}))\right)$ is a linear combination of $\zeta(\boldsymbol{u})$, where $\boldsymbol{u}$ runs over all replacements $(s_1\square\cdots\square s_n)$ of $\boldsymbol{s}$. Fix such a $\boldsymbol{u}$ with ${\rm type}(\boldsymbol{u})=(u_1,\ldots,u_m)$.
An intermediate replacement $\boldsymbol{t}$ contributing to $\zeta(\boldsymbol{u})$ is obtained by choosing, independently for each block of $\boldsymbol{u}$, a decomposition of that block into consecutive sub-blocks.
Hence the coefficient of $\zeta(\boldsymbol{u})$ equals $\prod_{i=1}^{m}S_{u_i}$, where
\[
S_N
=
\sum_{r=1}^{N}
\frac{(\beta)_r}{r!}
\sum_{\substack{j_1+\cdots+j_r=N\\ j_i\ge1}}
\prod_{i=1}^{r}
\frac{(\gamma)_{j_i}}{j_i!}.
\]
It therefore suffices to show that
\[
S_N=\frac{(\beta\gamma)_N}{N!}=\binom{\beta\gamma}{N}\qquad(N\ge1).
\]
We now consider the generating function
\[
F(x):=\sum_{N\ge1}S_Nx^N.
\]
Let $A(x)=(1+x)^\gamma-1$. Then, we have
\[
A(x)=\sum_{k\ge1}\frac{(\gamma)_k}{k!}x^k\quad\text{and}\quad  A(x)^r=\sum_{N\ge1}\sum_{\substack{j_1+\cdots+j_r=N\\ j_i\ge1}}\left(
\prod_{i=1}^{r}
\frac{(\gamma)_{j_i}}{j_i!}
\right)x^N.
\]
Exchanging the order of summation, we have
\begin{align*}
    F(x)&=\sum_{r\ge1}\frac{(\beta)_r}{r!}A(x)^r=\sum_{r\ge1}
\frac{(\beta)_r}{r!}
\bigl((1+x)^\gamma-1\bigr)^r.
\intertext{Adding the term corresponding to $r=0$, we have}
1+F(x)
&=
\sum_{r\ge0}
\frac{(\beta)_r}{r!}
\bigl((1+x)^\gamma-1\bigr)^r=
\left(1+(1+x)^\gamma-1\right)^{\beta}=(1+x)^{\beta\gamma}.
\end{align*}
Comparing the coefficients of $x^N$ gives $S_N=\binom{\beta\gamma}{N}$ for $N\ge1$, and this proves the theorem.
\end{proof}

\begin{cor}\label{cor:omega_inverse}
We have $\omega_1={\rm id}$. Consequently, for any $\beta\ne0$ the map $\omega_\beta$ is invertible with
\[\omega_\beta^{-1}=\omega_{\frac1\beta},\qquad\text{that is,}\qquad
\omega_\beta\circ\omega_{\frac{1}{\beta}}=\omega_{\frac{1}{\beta}}\circ\omega_\beta={\rm id}.\]
\end{cor}

\begin{remark}
In the language of Hoffman and Ihara \cite[\S3]{hi17}, a formal power series $f(x)=\sum_{i\ge1}c_ix^i$ induces the linear map $\Psi_f$ which replaces $i$ consecutive entries by their sum with the coefficient $c_i$, and $\Psi_f\Psi_g=\Psi_{f\circ g}$ \cite[Theorem~3.1]{hi17}. Our $\omega_\beta$ is $\Psi_{f_\beta}$ with $f_\beta(x)=\sum_{i\ge1}\binom{\beta}{i}x^i=(1+x)^\beta-1$,
and since $f_\beta\circ f_\gamma=(1+x)^{\beta\gamma}-1=f_{\beta\gamma}$, Theorem \ref{thm:omega_group} can also be deduced from \cite[Theorem~3.1]{hi17}.
For $\beta=-1$, the map $(-1)^r\omega_{-1}$ on the $\zeta(s_1,\ldots,s_r)$ of depth $r$ is the map $S$ of Ihara, Kajikawa, Ohno and Okuda \cite[Definition~1]{ikoo}, which translates multiple zeta-star values into multiple zeta values.
\end{remark}

\section{The duality for an arbitrary shape}\label{sec:duality}

In this section we prove that the duality holds for every skew shape $\lambda/\mu$. For $\lambda=(r,1)$ the duality can also be verified directly by a telescoping identity; see Appendix \ref{app:r1}.

Throughout, for partitions $\mu\subset\lambda$ we write $\pmb s$ for the tableau of shape $\lambda/\mu$ whose box $(i,j)$ carries the variable $s_{ij}$, and $\pmb s'$ for the tableau of shape $\lambda'/\mu'$ whose box $(i,j)$ carries $s_{ji}$.

\subsection{A skew version of Stanley's duality}

In this subsection, we write Stanley's formula as a sum over chains of partitions and show a skew version of Theorem \ref{thm:stanley33}. Let $\mu\subset\nu$ be partitions. If $\nu/\mu$ is a horizontal strip, we put
\[
\psi_{\nu/\mu}(\alpha)=\frac{\displaystyle \prod_{s\in\nu}B_{\nu,\mu}(s)}{\displaystyle \prod_{s\in\mu}C_{\nu,\mu}(s)},
\]
and we put $\psi_{\nu/\mu}(\alpha)=0$ otherwise. By a chain of length $r$ from $\mu$ to $\nu$ we mean a sequence
\[T:\ \mu=\kappa^{(0)}\subset\kappa^{(1)}\subset\cdots\subset\kappa^{(r)}=\nu\]
of partitions in which every step $\kappa^{(i)}/\kappa^{(i-1)}$ is a nonempty horizontal strip, and we put
\[w(T;\alpha)=\prod_{i=1}^{r}\psi_{\kappa^{(i)}/\kappa^{(i-1)}}(\alpha).\]
Thus a chain has $r$ steps and $r+1$ shapes. For a chain $T$ as above we list the numbers of boxes added at the steps,
\[\boldsymbol{c}(T)=\left(|\kappa^{(1)}/\kappa^{(0)}|,\ \ldots,\ |\kappa^{(r)}/\kappa^{(r-1)}|\right),\]
and also the variables sitting in the boxes added at the steps,
\[\pmb t_T=\left(\sum_{(i,j)\in\kappa^{(1)}/\kappa^{(0)}}s_{ij},\ \ \ldots,\ \sum_{(i,j)\in\kappa^{(r)}/\kappa^{(r-1)}}s_{ij}\right).\]

For a sequence $\boldsymbol{c}=(c_1,\ldots,c_r)$ of positive integers let $M_{\boldsymbol{c}}(x)=\sum_{i_1<\cdots<i_r}x_{i_1}^{c_1}\cdots x_{i_r}^{c_r}$ be the monomial quasi-symmetric function \cite{ge84}, with $M_{\emptyset}=1$. Discarding the empty steps as in Example \ref{ex:intro21}, Theorem \ref{thm:stanley63} says
\begin{equation}\label{eqn:Phi}
\Phi_{\nu/\mu}(x;\alpha):=\sum_{T}w(T;\alpha)M_{\boldsymbol{c}(T)}(x)=\frac{J_{\nu/\mu}(x;\alpha)}{j_\mu},
\end{equation}
the sum being over all chains from $\mu$ to $\nu$; that is, the weight in Theorem \ref{thm:stanley63} is $j_\mu w(T;\alpha)$. Since $j_\emptyset=1$, we have $\Phi_{\nu/\emptyset}=J_\nu$.

\begin{exam}\label{ex:chains21}
Let $\lambda=(2,1)$. For the four chains of Example \ref{ex:intro21}(2) we have
\[
\begin{array}{lll}
T & \boldsymbol{c}(T) & \pmb t_T\\\hline
\emptyset\subset(1)\subset(2)\subset(2,1) & (1,1,1) & (s_{11},s_{12},s_{21})\\
\emptyset\subset(1)\subset(1,1)\subset(2,1) & (1,1,1) & (s_{11},s_{21},s_{12})\\
\emptyset\subset(1)\subset(2,1) & (1,2) & (s_{11},s_{12}+s_{21})\\
\emptyset\subset(2)\subset(2,1) & (2,1) & (s_{11}+s_{12},s_{21})
\end{array}
\]
The first two chains have the same $\boldsymbol{c}(T)$ but different $\pmb t_T$. Accordingly, they give the two different terms $\zeta(s_{11},s_{12},s_{21})$ and $\zeta(s_{11},s_{21},s_{12})$ of $\zeta^J_{(2,1)}(\pmb s;\alpha)$, whereas in \eqref{eqn:Phi} they merge:
\[\Phi_{(2,1)}(x;\alpha)=6M_{(1,1,1)}+(\alpha+2)\left(M_{(1,2)}+M_{(2,1)}\right).\]
\end{exam}

\begin{exam}\label{ex:skew21}
Let $\mu=(1)$ and $\nu=(2,1)$. The chains from $(1)$ to $(2,1)$ are
\[(1)\subset(2,1),\qquad (1)\subset(2)\subset(2,1),\qquad (1)\subset(1,1)\subset(2,1),\]
with $\boldsymbol{c}(T)=(2),\ (1,1),\ (1,1)$ and $w(T;\alpha)=\alpha+2,\ \frac{2(\alpha+2)}{\alpha+1},\ \frac{2(2\alpha+1)}{\alpha+1}$ by the table of factors in Example \ref{ex:intro21}.
The chain of length $1$ occurs because $\nu/\mu=\{(1,2),(2,1)\}$ is itself a horizontal strip, and the last two weights add up to $6$. Thus,
\[\Phi_{(2,1)/(1)}(x;\alpha)=(\alpha+2)M_{(2)}+6M_{(1,1)}=3p_1^2+(\alpha-1)p_2.\]
\end{exam}

\begin{lemma}\label{lem:jtrans}
For every partition $\mu$ we have $j_{\mu'}(1/\alpha)=\alpha^{-2|\mu|}j_\mu(\alpha)$.
\end{lemma}

\begin{proof}
Writing $a(s)=a_\mu(s)$ and $l(s)=l_\mu(s)$, we have
\[j_\mu(\alpha)=\prod_{s\in\mu}\bigl(l(s)+\alpha(a(s)+1)\bigr)\bigl(l(s)+1+\alpha a(s)\bigr).\]
Conjugation interchanges $a$ and $l$, so the box $(j,i)$ of $\mu'$ contributes
\[\left(a(s)+\frac{l(s)+1}{\alpha}\right)\left(a(s)+1+\frac{l(s)}{\alpha}\right)
=\frac{1}{\alpha^2}\bigl(\alpha a(s)+l(s)+1\bigr)\bigl(\alpha a(s)+\alpha+l(s)\bigr)\]
to $j_{\mu'}(1/\alpha)$, which is $\alpha^{-2}$ times the contribution of $s$ to $j_\mu(\alpha)$.
\end{proof}

\begin{lemma}\label{lem:coprod}
For every partition $\nu$ and any two sets of variables $x$ and $y$ we have
\[J_\nu(x,y;\alpha)=\sum_{\mu\subset\nu}J_\mu(x;\alpha)\Phi_{\nu/\mu}(y;\alpha)
=\sum_{\mu\subset\nu}\frac{1}{j_\mu}J_\mu(x;\alpha)J_{\nu/\mu}(y;\alpha).\]
\end{lemma}

\begin{proof}
Since both sides are compatible with setting variables equal to $0$, it suffices to prove the identity when $x=(x_1,\ldots,x_p)$ and $y=(y_1,\ldots,y_q)$ are finite; we write $y_k=x_{p+k}$. By Theorem \ref{thm:stanley63}, $J_\nu(x,y;\alpha)$ is the sum of $w(T)x^T$ over the semi-standard Young  tableaux $T$ of shape $\nu$ with entries in $\{1,\ldots,p+q\}$.
The boxes of $T$ with entries at most $p$ form a partition $\mu\subset\nu$, and $T$ splits uniquely into a tableau $T_1$ of shape $\mu$ with entries in $\{1,\ldots,p\}$ and a tableau $T_2$ of shape $\nu/\mu$ with entries in $\{p+1,\ldots,p+q\}$.
Since $w(T)$ is a product over the steps of $T$ and $j_\emptyset=1$, we have $w(T)=w(T_1)\cdot w(T_2)/j_\mu$, where $w(T_2)$ is the weight of the skew tableau $T_2$ in Theorem \ref{thm:stanley63}. Summing over $T_2$ gives $J_{\nu/\mu}(y;\alpha)/j_\mu=\Phi_{\nu/\mu}(y;\alpha)$ by Theorem \ref{thm:stanley63} and \eqref{eqn:Phi},
and summing over $T_1$ gives $J_\mu(x;\alpha)$. Summing over $\mu$ gives both equalities.
\end{proof}

The following proposition is a skew version of Theorem \ref{thm:stanley33}. It is the Jack limit of the duality for skew Macdonald polynomials \cite[VI (7.16)]{mac}; we include a proof in the normalization of \eqref{eqn:Phi}.

\begin{prop}[{cf. \cite[VI (7.16)]{mac}}]\label{prop:skewdual}
Let $\mu\subset\nu$ be partitions and put $d=|\nu|-|\mu|$. Then
\[(-1)^{d}\omega_{-\frac{1}{\alpha}}\left(\Phi_{\nu'/\mu'}(x;1/\alpha)\right)
=\alpha^{-d}\Phi_{\nu/\mu}(x;\alpha),\]
or equivalently
\[(-1)^{d}\omega_{-\frac{1}{\alpha}}\left(J_{\nu'/\mu'}(x;1/\alpha)\right)
=\alpha^{-|\nu|-|\mu|}J_{\nu/\mu}(x;\alpha).\]
\end{prop}

\begin{proof}
Since $p_r(x,y)=p_r(x)+p_r(y)$ and $\omega_{-\frac{1}{\alpha}}$ is an algebra automorphism, applying $\omega_{-\frac{1}{\alpha}}$ to the set of variables $x\cup y$ is the same as applying it to $x$ and to $y$ separately. We apply $(-1)^{|\nu|}\omega_{-\frac{1}{\alpha}}$ to Lemma \ref{lem:coprod} for $\nu'$ at the parameter $1/\alpha$, that is, to
\[J_{\nu'}(x,y;1/\alpha)=\sum_{\tau\subset\nu'}\frac{1}{j_\tau(1/\alpha)}J_\tau(x;1/\alpha)J_{\nu'/\tau}(y;1/\alpha).\]
By Theorem \ref{thm:stanley33} the left-hand side becomes $\alpha^{-|\nu|}J_\nu(x,y;\alpha)$, which by Lemma \ref{lem:coprod} equals
\[\alpha^{-|\nu|}\sum_{\mu\subset\nu}\frac{1}{j_\mu(\alpha)}J_\mu(x;\alpha)J_{\nu/\mu}(y;\alpha).\]
On the right-hand side we write $\tau=\mu'$ and $(-1)^{|\nu|}=(-1)^{|\mu|}(-1)^{d}$. By using Theorem \ref{thm:stanley33} on each $J_{\mu'}(x;1/\alpha)$, it becomes
\[\sum_{\mu\subset\nu}\frac{\alpha^{-|\mu|}}{j_{\mu'}(1/\alpha)}J_\mu(x;\alpha)\cdot
(-1)^{d}\omega_{-\frac{1}{\alpha}}\left(J_{\nu'/\mu'}(y;1/\alpha)\right).\]
The Jack symmetric functions $J_\mu(x;\alpha)$, $\mu\subset\nu$, are linearly independent, so comparing their coefficients gives
\[(-1)^{d}\omega_{-\frac{1}{\alpha}}\left(J_{\nu'/\mu'}(y;1/\alpha)\right)
=\alpha^{|\mu|-|\nu|}\frac{j_{\mu'}(1/\alpha)}{j_\mu(\alpha)}J_{\nu/\mu}(y;\alpha)
=\alpha^{-|\nu|-|\mu|}J_{\nu/\mu}(y;\alpha)\]
by Lemma \ref{lem:jtrans}. Dividing both sides by $j_{\mu'}(1/\alpha)$ and using Lemma \ref{lem:jtrans} once more turns this into the assertion for $\Phi$.
\end{proof}

\subsection{The duality for Jack multiple zeta functions}
In this subsection, we compare the coefficients of both sides of the duality and reduce it to an identity for each skew shape.
\begin{definition}
For partitions $\mu\subset\nu$ and $r\ge0$ we set
\[A_r(\nu/\mu;\alpha)=\sum_{T}w(T;\alpha),\]
the sum being over the chains from $\mu$ to $\nu$ of length exactly $r$, so that $A_0(\nu/\mu;\alpha)=\delta_{\mu\nu}$ and $A_1(\nu/\mu;\alpha)=\psi_{\nu/\mu}(\alpha)$. We further set
\[\Sigma_{\nu/\mu}(z;\alpha)=\sum_{r\ge0}\binom{z}{r}A_r(\nu/\mu;\alpha)\in\mathbb{Q}(\alpha)[z].\]
\end{definition}

For an integer $q\ge0$, a chain of length $r$ together with a choice of $r$ of the $q$ variables corresponds to a monomial of $\Phi_{\nu/\mu}(x_1,\ldots,x_q;\alpha)$ whose support has $r$ elements. Hence
\begin{equation}\label{eqn:spec}
\Sigma_{\nu/\mu}(q;\alpha)=\Phi_{\nu/\mu}(1^q;\alpha)\qquad(q\in\mathbb{Z}_{\ge0}),
\end{equation}
so that the left-hand side is a polynomial in $q$ interpolating the right-hand side. We write $f(1^z)$ for the polynomial interpolating $f(1^q)$ in this way. Taking $q=1$ in \eqref{eqn:spec}, or $z=1$ in the definition of $\Sigma$, gives
\begin{equation}\label{eqn:sigma1}
\Sigma_{\nu/\mu}(1;\alpha)=A_1(\nu/\mu;\alpha)\qquad(\mu\ne\nu),
\end{equation}
since $\binom1r=0$ for $r\ge2$ and $A_0(\nu/\mu;\alpha)=0$.

\begin{exam}\label{ex:sigma21}
For $\mu=(1)$ and $\nu=(2,1)$, Example \ref{ex:skew21} gives $A_0=0$, $A_1=\alpha+2$ and $A_2=6$, so that
\[\Sigma_{(2,1)/(1)}(z;\alpha)=\binom z1(\alpha+2)+\binom z2\cdot6=z(\alpha+3z-1).\]
In accordance with \eqref{eqn:spec}, $M_{(2)}(1^q)=q$ and $M_{(1,1)}(1^q)=\binom q2$ give $\Phi_{(2,1)/(1)}(1^q;\alpha)=(\alpha+2)q+6\binom q2$, and in accordance with \eqref{eqn:sigma1}, $\Sigma_{(2,1)/(1)}(1;\alpha)=\alpha+2=A_1$.
\end{exam}

\begin{lemma}\label{lem:omspec}
For every $f\in\Lambda$ we have $\left(\omega_\beta f\right)(1^z)=f(1^{\beta z})$.
\end{lemma}

\begin{proof}
Both sides are linear in $f$, so we may take $f=p_\rho$ for a partition $\rho$. Then $\omega_\beta(p_\rho)=\beta^{\ell(\rho)}p_\rho$, while $p_r(1^q)=q$ for every $r\ge1$ and hence $p_\rho(1^z)=z^{\ell(\rho)}$.
Therefore $\left(\omega_\beta p_\rho\right)(1^z)=\beta^{\ell(\rho)}z^{\ell(\rho)}=p_\rho(1^{\beta z})$.
\end{proof}

\begin{lemma}\label{lem:local}
Let $\mu\subsetneq\nu$ be partitions and put $d=|\nu|-|\mu|$. Then
\[\Sigma_{\nu'/\mu'}\left(-\frac{1}{\alpha};\frac{1}{\alpha}\right)
=(-\alpha)^{-d}A_1(\nu/\mu;\alpha).\]
\end{lemma}

\begin{proof}
Specialize $x=(1^z)$ in Proposition \ref{prop:skewdual}. By Lemma \ref{lem:omspec} with $\beta=-\frac1\alpha$ and by \eqref{eqn:spec}, the left-hand side becomes $(-1)^{d}\Sigma_{\nu'/\mu'}\left(-\frac{z}{\alpha};\frac{1}{\alpha}\right)$, so that
\[(-1)^{d}\Sigma_{\nu'/\mu'}\left(-\frac{z}{\alpha};\frac{1}{\alpha}\right)
=\alpha^{-d}\Sigma_{\nu/\mu}(z;\alpha).\]
Both sides are polynomials in $z$. Putting $z=1$ and using \eqref{eqn:sigma1}, which applies since $\mu\ne\nu$, gives the assertion.
\end{proof}

\begin{exam}\label{ex:local}
Let $\mu=(1)$ and $\nu=(2,1)$, so that $\mu'=(1)$, $\nu'=(2,1)$ and $d=2$. Substituting $z=-\frac1\alpha$ and the parameter $\frac1\alpha$ in Example \ref{ex:sigma21},
\[\Sigma_{(2,1)/(1)}\left(-\frac1\alpha;\frac1\alpha\right)
=-\frac1\alpha\left(\frac1\alpha-\frac3\alpha-1\right)=\frac{\alpha+2}{\alpha^2}
=(-\alpha)^{-2}A_1\bigl((2,1)/(1);\alpha\bigr).\]
\end{exam}

Let $\mu\subset\lambda$ be partitions and let $d=|\lambda|-|\mu|$. For a sequence of partitions
\begin{equation}
\label{eq:ordered_set_partition}
\pmb u:\ \mu=\lambda^{(0)}\subsetneq\lambda^{(1)}\subsetneq\cdots\subsetneq\lambda^{(m)}=\lambda,
\end{equation}
where $\lambda^{(k)}/\lambda^{(k-1)}$ is not necessarily a horizontal strip, we write
\[\zeta(\pmb u)=\zeta\left(\sum_{(i,j)\in\lambda^{(1)}/\lambda^{(0)}}s_{ij},\ \ldots,\ \sum_{(i,j)\in\lambda^{(m)}/\lambda^{(m-1)}}s_{ij}\right).\]
The $\zeta(\pmb u)$ form a basis of the subspace $\mathfrak Z_{\lambda/\mu}$ of $\mathfrak Z$ which they span, and $\omega_\beta$ preserves $\mathfrak Z_{\lambda/\mu}$, since merging consecutive entries of $\zeta(\pmb u)$ corresponds to removing some of the $\lambda^{(k)}$ with $0<k<m$ from $\pmb u$. For $f\in\mathfrak Z_{\lambda/\mu}$ we write $[\zeta(\pmb u)]f$ for the coefficient of $\zeta(\pmb u)$ in the expansion of $f$ in this basis.
Both sides of the duality lie in $\mathfrak Z_{\lambda/\mu}$.

\begin{exam}\label{ex:osp21}
Let $\lambda=(2,1)$ and $\mu=\emptyset$.
\begin{enumerate}
\item For $\pmb u:\emptyset\subsetneq(1)\subsetneq(2,1)$ we have
$\zeta(\pmb u)=\zeta(s_{11},s_{12}+s_{21})$. This is $\zeta(\pmb t_T)$ for the third chain $T$ of Example \ref{ex:chains21}.
\item For $\pmb u:\emptyset\subsetneq(1,1)\subsetneq(2,1)$ we have
$\zeta(\pmb u)=\zeta(s_{11}+s_{21},s_{12})$, but $(1,1)/\emptyset$ is not a horizontal strip.
\end{enumerate}
Only the four chains $\pmb u$ in Example \ref{ex:chains21} have a nonzero coefficient in $\zeta^J_{(2,1)}(\pmb s;\alpha)$.
\end{exam}

\begin{lemma}\label{lem:factor}
Let $\pmb u$ be a sequence of partitions as in \eqref{eq:ordered_set_partition}. Then
\begin{align}
[\zeta(\pmb u)]\,\zeta^J_{\lambda/\mu}(\pmb s;\alpha)
&=j_\mu(\alpha)\prod_{k=1}^{m}A_1\!\left(\lambda^{(k)}/\lambda^{(k-1)};\alpha\right),
\label{eqn:rhscoeff}\\
[\zeta(\pmb u)]\,\omega_{-\frac{1}{\alpha}}\!\left(\zeta^J_{\lambda'/\mu'}\left(\pmb s';\frac1\alpha\right)\right)
&=j_{\mu'}\!\left(\frac1\alpha\right)\prod_{k=1}^{m}\Sigma_{\lambda^{(k)\prime}/\lambda^{(k-1)\prime}}\!\left(-\frac1\alpha;\frac1\alpha\right).
\label{eqn:lhscoeff}
\end{align}
\end{lemma}

\begin{proof}
We first prove \eqref{eqn:rhscoeff}. By Definition \ref{def:jackzeta}, grouping the tableaux $M$ according to the chain $T$ they determine, $\zeta^J_{\lambda/\mu}(\pmb s;\alpha)$ is the sum of $j_\mu(\alpha)w(T;\alpha)\zeta(\pmb t_T)$ over the chains $T$ from $\mu$ to $\lambda$ (note that $w(M)=j_\mu(\alpha)w(T;\alpha)$ by Theorem \ref{thm:stanley63}).
If every $\lambda^{(k)}/\lambda^{(k-1)}$ is a horizontal strip, then $\pmb u$ is itself a chain and $\zeta(\pmb u)=\zeta(\pmb t_T)$ only for $T=\pmb u$, so that the coefficient of $\zeta(\pmb u)$ is $j_\mu(\alpha)w(\pmb u;\alpha)=j_\mu(\alpha)\prod_k\psi_{\lambda^{(k)}/\lambda^{(k-1)}}(\alpha)$.
Otherwise the coefficient of $\zeta(\pmb u)$ is $0$, and so is $\prod_k\psi_{\lambda^{(k)}/\lambda^{(k-1)}}(\alpha)$. This proves \eqref{eqn:rhscoeff}.

We next prove \eqref{eqn:lhscoeff}. By Definition \ref{def:jackzeta},
\[\zeta^J_{\lambda'/\mu'}\left(\pmb s';\frac1\alpha\right)=j_{\mu'}\!\left(\frac1\alpha\right)\sum_{T'}w\left(T';\frac1\alpha\right)\zeta(\pmb t_{T'}),
\quad T':\ \mu'=\kappa^{(0)}\subset\cdots\subset\kappa^{(\ell)}=\lambda'.\]
Since $\pmb s'$ carries the variables transposed, the $i$-th entry of $\pmb t_{T'}$ is the sum of the $s_{ij}$ over the boxes of $\kappa^{(i)\prime}/\kappa^{(i-1)\prime}$, which is a vertical strip of $\lambda/\mu$.
As $\omega_{-\frac1\alpha}$ merges consecutive entries, the coefficient of $\zeta(\pmb u)$ is $j_{\mu'}(\frac1\alpha)$ times
\begin{equation}\label{eqn:lhssum}
\sum_{T'}w\left(T';\frac1\alpha\right)\prod_{k=1}^{m}\binom{-1/\alpha}{r_k},
\end{equation}
where $T'$ runs over the chains whose transposed sequence of steps refines $\pmb u$ and $r_k$ is the number of steps of $T'$ inside $\lambda^{(k)\prime}/\lambda^{(k-1)\prime}$. Here we say that the transposed sequence of steps refines $\pmb u$ if there are indices $0=i_0<\cdots<i_m=\ell$ with $\kappa^{(i_k)}=\lambda^{(k)\prime}$ for all $k$. Then $r_k=i_k-i_{k-1}$. Therefore, such a $T'$ corresponds bijectively to the $m$-tuple $(T_1',\ldots,T_m')$ of chains
\[T_k':\ \lambda^{(k-1)\prime}=\kappa^{(i_{k-1})}\subset\cdots\subset\kappa^{(i_k)}=\lambda^{(k)\prime},\]
of lengths $r_1,\ldots,r_m$, and every step of $T'$ belongs to exactly one $T_k'$. Regrouping the two products in \eqref{eqn:lhssum} along this partition of the steps,
\begin{align*}
\prod_{i=1}^{\ell}\psi_{\kappa^{(i)}/\kappa^{(i-1)}}\left(\frac1\alpha\right)
\cdot\prod_{k=1}^{m}\binom{-1/\alpha}{r_k}
&=\prod_{k=1}^{m}\left(\binom{-1/\alpha}{r_k}
\prod_{i=i_{k-1}+1}^{i_k}\psi_{\kappa^{(i)}/\kappa^{(i-1)}}\left(\frac1\alpha\right)\right)\\
&=\prod_{k=1}^{m}\binom{-1/\alpha}{r_k}w\left(T_k';\frac1\alpha\right).
\end{align*}
Since the $k$-th factor depends only on $T_k'$, we obtain
\begin{align*}
\eqref{eqn:lhssum}&=\prod_{k=1}^{m}\left(\sum_{r\ge0}\binom{-1/\alpha}{r}
\sum_{T_k'}w\left(T_k';\frac1\alpha\right)\right)\\
&=\prod_{k=1}^{m}\left(\sum_{r\ge0}\binom{-1/\alpha}{r}
A_r\!\left(\lambda^{(k)\prime}/\lambda^{(k-1)\prime};\frac1\alpha\right)\right),
\end{align*}
the inner sum being over the chains $T_k'$ from $\lambda^{(k-1)\prime}$ to $\lambda^{(k)\prime}$ of length $r$. Multiplying by $j_{\mu'}(\frac1\alpha)$, we obtain \eqref{eqn:lhscoeff}.
\end{proof}

\begin{exam}\label{ex:factor21}
Let $\lambda=(2,1)$, $\mu=\emptyset$ and $\pmb u:\emptyset\subsetneq(1)\subsetneq(2,1)$ as in Example \ref{ex:osp21}(1), so that $m=2$. By \eqref{eqn:rhscoeff} and Example \ref{ex:chains21},
\[[\zeta(\pmb u)]\,\zeta^J_{(2,1)}(\pmb s;\alpha)=A_1\bigl((1)/\emptyset;\alpha\bigr)A_1\bigl((2,1)/(1);\alpha\bigr)=1\cdot(\alpha+2).\]
For \eqref{eqn:lhscoeff} note that $\lambda'=(2,1)$. The chains $T'$ in $\lambda'$ whose transposed sequence of steps refines $\pmb u$ are those whose first step is the box $(1,1)$; the remaining two boxes are then added in one step or in two steps. With their weights at $\frac1\alpha$ and the factors of \eqref{eqn:lhssum} they are
\[
\begin{array}{llll}
T' & w(T';\frac1\alpha) & (r_1,r_2) & \prod_k\binom{-1/\alpha}{r_k}\\\hline
\emptyset\subset(1)\subset(2,1) & \frac{1+2\alpha}{\alpha} & (1,1) & \frac{1}{\alpha^2}\\[6pt]
\emptyset\subset(1)\subset(2)\subset(2,1) & \frac{2(1+2\alpha)}{1+\alpha} & (1,2) & -\frac{1+\alpha}{2\alpha^3}\\[6pt]
\emptyset\subset(1)\subset(1,1)\subset(2,1) & \frac{2(2+\alpha)}{1+\alpha} & (1,2) & -\frac{1+\alpha}{2\alpha^3}
\end{array}
\]
so that \eqref{eqn:lhssum} equals $\frac{1+2\alpha}{\alpha^3}-\frac{1+2\alpha}{\alpha^3}-\frac{2+\alpha}{\alpha^3}=-\frac{\alpha+2}{\alpha^3}$. On the other hand the right-hand side of \eqref{eqn:lhscoeff} is
\[\Sigma_{(1)}\left(-\frac1\alpha;\frac1\alpha\right)\Sigma_{(2,1)/(1)}\left(-\frac1\alpha;\frac1\alpha\right)
=\left(-\frac1\alpha\right)\cdot\frac{\alpha+2}{\alpha^2}\]
by Example \ref{ex:local}. Multiplying by $(-1)^3$ gives $\alpha^{-3}(\alpha+2)$, the coefficient found above, as Theorem \ref{thm:general} asserts.
\end{exam}
Combining Lemmas \ref{lem:local} and \ref{lem:factor}, we obtain the duality for every skew shape.
\begin{theorem}\label{thm:general}
For any partitions $\mu\subset\lambda$, we have
\[(-1)^{|\lambda|-|\mu|}\omega_{-\frac{1}{\alpha}}\left(\zeta_{\lambda'/\mu'}^J\left(\pmb s';\frac{1}{\alpha}\right)\right)
=\frac{1}{\alpha^{|\lambda|+|\mu|}}\zeta_{\lambda/\mu}^J\left(\pmb s;\alpha\right).\]
In particular, for every partition $\lambda$ of $n$, we have
\[(-1)^{n}\omega_{-\frac{1}{\alpha}}\left(\zeta_{\lambda'}^J\left(\pmb s';\frac{1}{\alpha}\right)\right)
=\frac{1}{\alpha^{n}}\zeta_{\lambda}^J\left(\pmb s;\alpha\right).\]
\end{theorem}

\begin{proof}
Fix a sequence of partitions $\pmb u$ as in \eqref{eq:ordered_set_partition}. Applying Lemma \ref{lem:local} with $\mu=\lambda^{(k-1)}$ and $\nu=\lambda^{(k)}$ for $k=1,\ldots,m$ and multiplying the resulting $m$ identities, we obtain from \eqref{eqn:rhscoeff} and \eqref{eqn:lhscoeff} that
\begin{align*}
[\zeta(\pmb u)]\,\omega_{-\frac{1}{\alpha}}\left(\zeta^J_{\lambda'/\mu'}\left(\pmb s';\frac1\alpha\right)\right)
&=j_{\mu'}\!\left(\frac1\alpha\right)\prod_{k=1}^{m}(-\alpha)^{-|\lambda^{(k)}/\lambda^{(k-1)}|}A_1\!\left(\lambda^{(k)}/\lambda^{(k-1)};\alpha\right)\\
&=(-\alpha)^{-d}\alpha^{-2|\mu|}[\zeta(\pmb u)]\,\zeta^J_{\lambda/\mu}(\pmb s;\alpha),
\end{align*}
because $\sum_{k=1}^{m}|\lambda^{(k)}/\lambda^{(k-1)}|=d$ and $j_{\mu'}(\frac1\alpha)=\alpha^{-2|\mu|}j_\mu(\alpha)$ by Lemma \ref{lem:jtrans}. Since $d+2|\mu|=|\lambda|+|\mu|$, multiplying by $(-1)^{d}$ and letting $\pmb u$ run over the basis of $\mathfrak Z_{\lambda/\mu}$ gives the first assertion.
The second assertion is the case $\mu=\emptyset$.
\end{proof}

\begin{cor}\label{cor:alpha1}
Let $\lambda$ be a partition of $n$. For every chain $T$ from $\emptyset$ to $\lambda$ we have $w(T;1)=h_\lambda$, so that $\zeta^J_\lambda(\pmb s;1)=h_\lambda\zeta_\lambda(\pmb s)$. Consequently, Theorem \ref{thm:general} with $\mu=\emptyset$ specializes at $\alpha=1$ to
\[(-1)^{n}\omega_{-1}\left(\zeta_{\lambda'}(\pmb s')\right)=\zeta_{\lambda}(\pmb s)\]
for the Schur multiple zeta functions.
\end{cor}

\begin{proof}
For $\alpha=1$ the two cases in the definition of $B_{\nu,\mu}(s)$ give the same value $\nu'_j-i+\nu_i-j+1=h_\nu(s)$, the hook length of $s=(i,j)$ in $\nu$, and likewise $C_{\nu,\mu}(s)=h_\mu(s)$.
Therefore, $\psi_{\nu/\mu}(1)=h_\nu/h_\mu$ for every horizontal strip $\nu/\mu$, and the product over the steps of $T$ telescopes to $h_\lambda/h_\emptyset=h_\lambda$.
The remaining assertions follow from Definition \ref{def:jackzeta}, the relation $h_\lambda=h_{\lambda'}$ and Theorem \ref{thm:general}.
\end{proof}

\begin{remark}\label{rem:general_f}
The proof of Theorem \ref{thm:general} uses only the linear independence of the $\zeta(\pmb u)$, so that we may apply any linear map on $\mathfrak Z_{\lambda/\mu}$ to it. For example, we can combine Stanley's weights with the generalization of the Schur functions considered by Bachmann and Charlton \cite{bc}.
Let $A$ be a commutative $\mathbb{Q}(\alpha)$-algebra, let $f:\mathbb{Z}_{>0}\times S\to A$ be any map, and let $N\ge1$. For partitions $\mu\subset\lambda$ and variables $\pmb s=(s_{ij})$ on $\lambda/\mu$, we put
\[J^f_{\lambda/\mu}(\pmb s;\alpha)=\sum_{M}w(M)\prod_{(i,j)\in\lambda/\mu}f(m_{ij},s_{ij}),\]
the sum being over all $M\in{\rm SSYT}(\lambda/\mu)$ with entries at most $N$. For $f(m,s)=x_m$ this is $J_{\lambda/\mu}(x_1,\ldots,x_N;\alpha)$ by Theorem \ref{thm:stanley63}, for $f(m,s)=m^{-s}$ it tends to $\zeta^J_{\lambda/\mu}(\pmb s;\alpha)$ as $N\to\infty$, and for $\mu=\emptyset$ and $\alpha=1$ it is $h_\lambda$ times the function of Bachmann and Charlton \cite{bc}.
Then Theorem \ref{thm:general} can be generalized to $J^f_{\lambda/\mu}(\pmb s;\alpha)$. For $f(m,s)=x_m$, the generalized identity is Proposition \ref{prop:skewdual}. Thus Theorem \ref{thm:general} is a refinement of Stanley's duality which contains it as a special case, although our proof uses the latter.
\end{remark}

\section{Expansions in terms of Schur multiple zeta functions and the $(2^i,1^j)$-formula}\label{sec:2i1j}
In this section, we show that any expansion of a Jack symmetric function in terms of the Schur functions lifts to the Jack multiple zeta functions after a symmetrization of the variables, and we apply this to the following theorem.
\begin{theorem}[{\cite[(47)]{st89}}, {\cite[Theorem~1.1]{ko93}}]\label{thm:stanley47}
    \[J_{(2^i,1^j)}=\sum_{r=0}^i(i)_r(\alpha+i+j)_r(i-r-\alpha)_{i-r}(i+j-r)!s_{(2^r,1^{2(i-r)+j})}.\]
\end{theorem}
Stanley proved the monomial expansion of $J_{(2^i,1^j)}$ \cite[(46)]{st89} and stated this formula as a conjecture \cite[(47)]{st89}, expecting that it could be deduced from the monomial expansion.
A complete proof was given by Koike \cite[Theorem~1.1]{ko93}, who computed the Kostka numbers $K_{(2^r,1^{2(i-r)+j}),\,(2^k,1^{2(i-k)+j})}$ by the hook length formula and reduced the statement to a binomial identity.

We start with an example which shows why a symmetrization of the variables is needed.

\begin{exam}\label{ex:21naive}
Let $\lambda=(2,1)$. By Example \ref{ex:intro21}(2), the four chains from $\emptyset$ to $(2,1)$ give
\begin{align*}
\zeta^{J}_{(2,1)}\left(\ytableaushort{{s_{11}}{s_{12}},{s_{21}}};\alpha\right)&=\frac{2(\alpha+2)}{\alpha+1}\zeta(s_{11},s_{12},s_{21})+\frac{2(2\alpha+1)}{\alpha+1}\zeta(s_{11},s_{21},s_{12})\\
&\qquad+(\alpha+2)\zeta(s_{11},s_{12}+s_{21})+(\alpha+2)\zeta(s_{11}+s_{12},s_{21})\\
&=-\frac{\alpha^2+\alpha-2}{\alpha+1}\zeta_{(1,1,1)}\left(\ytableaushort{{s_{11}},{s_{12}},{s_{21}}}\right)-\frac{\alpha^2-\alpha}{\alpha+1}\zeta_{(1,1,1)}\left(\ytableaushort{{s_{11}},{s_{21}},{s_{12}}}\right)\\
&\qquad+(\alpha+2)\zeta_{(2,1)}\left(\ytableaushort{{s_{11}}{s_{12}},{s_{21}}}\right),
\intertext{and if $s_{12}=s_{21}$ or $\alpha=1$, this becomes}
&=2(1-\alpha)\zeta_{(1,1,1)}\left(\ytableaushort{{s_{11}},{s_{12}},{s_{21}}}\right)+(\alpha+2)\zeta_{(2,1)}\left(\ytableaushort{{s_{11}}{s_{12}},{s_{21}}}\right).
\end{align*}
Thus the expansion $J_{(2,1)}=(\alpha+2)s_{(2,1)}-2(\alpha-1)s_{(1,1,1)}$ lifts to $\zeta^J_{(2,1)}$ only under the condition $s_{12}=s_{21}$ or $\alpha=1$, because the first two chains have different weights.
\end{exam}

Throughout this section we fix an integer $n\ge1$ and a tuple of variables $\pmb s=(s_1,\ldots,s_n)$. For a partition $\lambda$ of $n$, by a labeling of $\lambda$ we mean a bijection $\phi$ from the set of boxes of $\lambda$ onto $\{1,\ldots,n\}$ with $\phi((1,1))=1$, and we write $\pmb s^{\phi}$ for the tableau of shape $\lambda$ carrying $s_{\phi((i,j))}$ in the box $(i,j)$. Let $\mathfrak S_{n-1}$ denote the symmetric group on $\{2,\ldots,n\}$, acting on labelings by $\sigma\cdot\phi=\sigma\circ\phi$. We put
\[\sum_{\rm sym}\zeta_\lambda(\pmb s):=\sum_{\sigma\in\mathfrak S_{n-1}}\zeta_\lambda\left(\pmb s^{\sigma\phi}\right),
\qquad
\sum_{\rm sym}\zeta_\lambda^{J}(\pmb s;\alpha):=\sum_{\sigma\in\mathfrak S_{n-1}}\zeta_\lambda^{J}\left(\pmb s^{\sigma\phi};\alpha\right).\]
In other words, ${\rm sym}$ means that we symmetrize over all the variables except the one sitting in the box $(1,1)$. Any two labelings of $\lambda$ differ by an element of $\mathfrak S_{n-1}$, so these sums are independent of the choice of $\phi$. In particular they are simultaneously defined for all partitions of $n$, which is what makes it possible to compare different shapes below.

\begin{definition}
For a composition $\boldsymbol{c}=(c_1,\ldots,c_\ell)$ of $n$ we set
\[O_{\boldsymbol{c}}(\pmb s):=\sum_{(I_1,\ldots,I_\ell)}\zeta\left(\sum_{a\in I_1}s_a,\ \ldots,\ \sum_{a\in I_\ell}s_a\right),\]
where $(I_1,\ldots,I_\ell)$ runs over all ordered set partitions of $\{1,\ldots,n\}$ with $|I_k|=c_k$ for every $k$ and $1\in I_1$.
\end{definition}

For $\boldsymbol{c}=(c_1,\ldots,c_\ell)$ we write
\[e(\boldsymbol{c})=(c_1-1)!\,c_2!\cdots c_\ell!,\]
and for a symmetric function $f$ we denote by $[x^{\boldsymbol{c}}]f$ the coefficient of $x_1^{c_1}x_2^{c_2}\cdots x_\ell^{c_\ell}$ in $f$.

\begin{lemma}\label{lem:orbit}
For every partition $\lambda$ of $n$ we have
\[\sum_{\rm sym}\zeta_\lambda(\pmb s)=\sum_{\boldsymbol{c}}e(\boldsymbol{c})\left([x^{\boldsymbol{c}}]s_\lambda\right)O_{\boldsymbol{c}}(\pmb s),
\qquad
\sum_{\rm sym}\zeta_\lambda^{J}(\pmb s;\alpha)=\sum_{\boldsymbol{c}}e(\boldsymbol{c})\left([x^{\boldsymbol{c}}]J_\lambda\right)O_{\boldsymbol{c}}(\pmb s),\]
where $\boldsymbol{c}$ runs over all compositions of $n$.
\end{lemma}

\begin{proof}
Grouping the entries of a semi-standard Young  tableau $M$ of shape $\lambda$ according to their values, $M$ is determined by a chain
\[T:\ \emptyset=\lambda^{(0)}\subset\lambda^{(1)}\subset\cdots\subset\lambda^{(\ell)}=\lambda\]
together with the strictly increasing sequence of the values attained. Writing $I_k=\phi\left(\lambda^{(k)}/\lambda^{(k-1)}\right)$ and $\pmb t_T=\left(\sum_{a\in I_1}s_a,\ldots,\sum_{a\in I_\ell}s_a\right)$, we therefore have
\[\zeta_\lambda^{J}\left(\pmb s^{\phi};\alpha\right)=\sum_{T}w(T;\alpha)\zeta(\pmb t_T),
\qquad
\zeta_\lambda\left(\pmb s^{\phi}\right)=\sum_{T}\zeta(\pmb t_T).\]
Since $(1,1)\in\lambda^{(1)}$, we always have $1\in I_1$. Suppose that $\boldsymbol{c}(T)=\boldsymbol{c}$. Then the orbit of $(I_1,\ldots,I_\ell)$ under $\mathfrak S_{n-1}$ consists exactly of the ordered set partitions occurring in $O_{\boldsymbol{c}}(\pmb s)$, and its stabilizer is $\mathfrak S_{I_1\setminus\{1\}}\times\mathfrak S_{I_2}\times\cdots\times\mathfrak S_{I_\ell}$, of order $e(\boldsymbol{c})$. Consequently
\[\sum_{\sigma\in\mathfrak S_{n-1}}\zeta\left(\sigma\pmb t_T\right)=e(\boldsymbol{c})O_{\boldsymbol{c}}(\pmb s).\]
Finally, by Theorem \ref{thm:stanley63} the chains $T$ with $\boldsymbol{c}(T)=\boldsymbol{c}$ contribute $\sum_{T\,:\,\boldsymbol{c}(T)=\boldsymbol{c}}w(T;\alpha)=[x^{\boldsymbol{c}}]J_\lambda$, while their number is the Kostka number $[x^{\boldsymbol{c}}]s_\lambda$ (see \cite[Chapter~I]{mac}). Summing over $T$ gives the assertion.
\end{proof}

\begin{theorem}\label{thm:sym_transfer}
Let $\lambda$ be a partition of $n$ and let
\[J_\lambda(x;\alpha)=\sum_{\mu}c_{\lambda\mu}(\alpha)s_\mu(x)\]
be the expansion of $J_\lambda$ in the Schur basis. Then
\[\sum_{\rm sym}\zeta_\lambda^{J}(\pmb s;\alpha)=\sum_{\mu}c_{\lambda\mu}(\alpha)\sum_{\rm sym}\zeta_\mu(\pmb s).\]
\end{theorem}

\begin{proof}
By Lemma \ref{lem:orbit}, the left-hand side equals $\sum_{\boldsymbol{c}}e(\boldsymbol{c})\left([x^{\boldsymbol{c}}]J_\lambda\right)O_{\boldsymbol{c}}(\pmb s)$ and the right-hand side equals $\sum_{\boldsymbol{c}}e(\boldsymbol{c})\left(\sum_{\mu}c_{\lambda\mu}(\alpha)[x^{\boldsymbol{c}}]s_\mu\right)O_{\boldsymbol{c}}(\pmb s)$.
The two coefficients agree for every $\boldsymbol{c}$ because $[x^{\boldsymbol{c}}]J_\lambda=\sum_\mu c_{\lambda\mu}(\alpha)[x^{\boldsymbol{c}}]s_\mu$.
\end{proof}

\begin{remark}
Summation over permutations of the variables was already used for the Pieri formulas in \cite[Theorem~5.3]{nt}, and expansions of products of Schur multiple zeta functions obtained by symmetrizing the variables appear in \cite{nak,ha1}; see \cite{ha2} for the Schur $P$- and $Q$-multiple zeta functions.
\end{remark}

Combining Theorem \ref{thm:sym_transfer} with Theorem \ref{thm:stanley47}, we obtain the following formula in the case $\lambda=(2^i,1^j)$.

\begin{cor}\label{cor:2i1j}
Let $i,j\ge0$, let $\lambda=(2^i,1^j)$ and $n=2i+j$. Then we have
\begin{align}
\label{for:2^i1^j}
\sum_{\rm sym}\zeta_{(2^i,1^j)}^{J}(\pmb s;\alpha)
&=\sum_{r=0}^i(i)_r(\alpha+i+j)_r(i-r-\alpha)_{i-r}(i+j-r)!
\sum_{\rm sym}\zeta_{(2^r,1^{2(i-r)+j})}(\pmb s).
\end{align}
\end{cor}

\section{The one-row case}\label{sec:omega}
In this section, we consider the case $\lambda=(r)$.
\subsection{Interpolation and the values at $\{s\}^r$}

In this subsection, we show that the one-row Jack multiple zeta functions interpolate the multiple zeta and zeta-star functions, and we compute their values at $\{s\}^r$ by a generating function. For an integer $t\ge0$ we define
\[g_\alpha(t)=\frac{1}{t!}\prod_{j=1}^{t-1}(1+\alpha j)\quad(t\ge1),\qquad g_\alpha(0)=1,\]
so that $g_\alpha(1)=1$.

\begin{lemma}\label{lem:1r}For any $r\ge1$, we have
\[
\ytableausetup{boxsize=normal}
\frac{1}{r!}\zeta_{(r)}^J\left(\ytableaushort{{s_{1}}{\cdots}{s_{r}}};\alpha\right)
=\sum_{\substack{\boldsymbol{t}=(s_1 \square s_2 \square \cdots \square s_r)\\{\rm type}(\boldsymbol{t})=(t_1,\ldots,t_\ell)}}
\prod_{i=1}^{\ell}g_\alpha(t_i)\,\zeta(\boldsymbol{t})
\]
and
\[
\zeta_{(1^r)}^J\left(\ytableaushort{{s_{1}},{\vdots},{s_{r}}};\alpha\right)=r!\,\zeta(s_1,s_2,\ldots,s_r).
\]
\end{lemma}
\begin{proof}
    By the definition of $\zeta_\lambda^J$, we have
    \begin{equation}
        \zeta_{(r)}^J\left(\ytableaushort{{s_{1}}{\cdots}{s_{r}}};\alpha\right)=\sum_{M}\frac{w(M)}{m_1^{s_1}\cdots m_r^{s_r}}.
        \label{eqn:r_zeta}
    \end{equation}
   We now compute $w(M)$. A semi-standard Young  tableau of shape $(r)$ is simply a weakly increasing sequence
\[
m_1\le m_2\le\cdots\le m_r.
\]
We first write the equalities in the weakly increasing sequence as blocks as in \eqref{type}. There exist integers $n_i,t_i$ with $1\le n_1<n_2<\cdots<n_\ell$ and $t_1+\cdots+t_\ell=r$ such that $(m_1,\dots,m_r)=(\{n_1\}^{t_1},\{n_2\}^{t_2},\dots,\{n_\ell\}^{t_{\ell}})$, where $\{n\}^t$ means that $n$ is repeated $t$ times. Letting
\[\boldsymbol{t}=(s_1+s_2+\cdots+s_{t_1},s_{t_1+1}+\cdots+s_{t_1+t_2},\ldots,s_{r-t_\ell+1}+\cdots+s_r),\]
we note that ${\rm type}(\boldsymbol{t})=(t_1,\ldots,t_\ell)$. When $\lambda=(k)$ and $\mu=(k-t)$, we have
\[
\frac{\displaystyle \prod_{s\in\lambda}B_{\lambda,\mu}(s)}{\displaystyle \prod_{s\in\mu}C_{\lambda,\mu}(s)}=\frac{\displaystyle \prod_{j=1}^{k-t}\alpha(k-j+1)\prod_{j=k-t+1}^{k}(1+\alpha(k-j))}{\displaystyle \prod_{j=1}^{k-t}\alpha(k-t-j+1)}=\binom{k}{t}\prod_{j=0}^{t-1}(1+\alpha j).
\]
Therefore, we have
\[
w(M)
=
\frac{r!}{t_1!\cdots t_\ell!}\;
\prod_{i=1}^\ell \prod_{j=0}^{t_i-1}(1+\alpha j)
=r!\prod_{i=1}^{\ell}g_\alpha(t_i),
\]
since the factor with $j=0$ is $1$. Putting this formula in \eqref{eqn:r_zeta} and collecting the tableaux $M$ with the same $(t_1,\ldots,t_\ell)$, we obtain the first assertion.

We obtain the second assertion, since only the chain $\emptyset\subset(1)\subset(1,1)\subset\cdots\subset(1^r)$ appears.
\end{proof}

\begin{cor}\label{cor:interp}
For any $r\ge1$, we have
\begin{align*}
\frac{1}{r!}\zeta_{(r)}^J\left(\ytableaushort{{s_{1}}{\cdots}{s_{r}}};1\right)&=\zetas(s_1,s_2,\ldots,s_r),\\
\frac{1}{r!}\zeta_{(r)}^J\left(\ytableaushort{{s_{1}}{\cdots}{s_{r}}};-1\right)&=\zeta(s_1,s_2,\ldots,s_r).
\end{align*}
\end{cor}

\begin{proof}
For $\alpha=1$ we have $g_1(t)=\frac{1}{t!}\prod_{j=1}^{t-1}(1+j)=1$ for every $t$, so that the right-hand side of the first formula of Lemma \ref{lem:1r} is the sum of $\zeta(\boldsymbol{t})$ over all replacements of $\boldsymbol{s}$, which is $\zetas(s_1,\ldots,s_r)$.
For $\alpha=-1$ the factor with $j=1$ vanishes, so $g_{-1}(t)=0$ for $t\ge2$ and only the term $\boldsymbol{t}=(s_1,\ldots,s_r)$ appears.
\end{proof}

\begin{remark}
This interpolation is different from the interpolation of Yamamoto \cite{y13}, which we write as $\zeta^{c}$ with the parameter $c$ in order to avoid a conflict of notation; its coefficient of $\zeta(\boldsymbol{t})$ is $c^{r-\ell}=\prod_i c^{t_i-1}$. The two families agree exactly at the two ends, the multiple zeta function ($\alpha=-1$,
$c=0$) and the multiple zeta-star function ($\alpha=1$, $c=1$), and differ for $\alpha\ne\pm1$. In terms of the maps $\Psi_f$ of \cite{hi17}, Yamamoto's family is obtained from $\zeta(s_1,\ldots,s_r)$ by $\Psi_f$ with $f(x)=x/(1-cx)$, whereas ours is obtained with $f(x)=\sum_{i\ge1}g_\alpha(i)x^i=(1-\alpha x)^{-1/\alpha}-1$.
\end{remark}

\begin{remark}
Since $g_\alpha(t)=(-\alpha)^t\binom{-1/\alpha}{t}$, the first formula of Lemma \ref{lem:1r} can be written as
\[\frac{1}{r!}\zeta_{(r)}^J(\pmb s;\alpha)=(-\alpha)^r\omega_{-\frac1\alpha}\bigl(\zeta(s_1,\ldots,s_r)\bigr).\]
By the second formula of Lemma \ref{lem:1r}, this is exactly Theorem \ref{thm:general} for $\lambda=(r)$. Therefore, Theorem \ref{thm:general} gives an alternative proof of the first formula.
\end{remark}

We next specialize all the variables to a common value. For an integer $s\ge2$ we write $\{s\}^r$ for the tableau of shape $(r)$ all of whose entries are $s$.

\begin{prop}\label{prop:genfun}
Let $s\ge2$ be an integer and $\alpha\ne0$. For $|x|$ small enough we have
\[
\sum_{r\ge0}\frac{1}{r!}\zeta_{(r)}^J\left(\{s\}^r;\alpha\right)x^r
=\prod_{n\ge1}\left(1-\frac{\alpha x}{n^s}\right)^{-\frac1\alpha},
\]
the term $r=0$ being $1$.
\end{prop}

\begin{proof}
Since $g_\alpha(t)=(-\alpha)^t\binom{-1/\alpha}{t}$ for every $t\ge0$, we have
\[\sum_{t\ge0}g_\alpha(t)y^t=\sum_{t\ge0}\binom{-1/\alpha}{t}(-\alpha y)^t=(1-\alpha y)^{-\frac1\alpha}.\]
On the other hand, by Lemma \ref{lem:1r} with all the variables equal to $s$,
\begin{align*}
\frac{1}{r!}\zeta_{(r)}^J\left(\{s\}^r;\alpha\right)
&=\sum_{\ell\ge1}\ \sum_{\substack{t_1+\cdots+t_\ell=r\\ t_i\ge1}}\ \prod_{i=1}^{\ell}g_\alpha(t_i)\ \zeta(st_1,\ldots,st_\ell),\\
\zeta(st_1,\ldots,st_\ell)&=\sum_{n_1<\cdots<n_\ell}\ \prod_{i=1}^{\ell}\frac{1}{n_i^{st_i}}.
\end{align*}
Multiplying by $x^r$, summing over $r\ge0$ and exchanging the order of summation, we obtain
\[
\sum_{r\ge0}\frac{x^r}{r!}\zeta_{(r)}^J\left(\{s\}^r;\alpha\right)
=\sum_{\ell\ge0}\ \sum_{n_1<\cdots<n_\ell}\ \prod_{i=1}^{\ell}\left(\sum_{t\ge1}g_\alpha(t)\left(\frac{x}{n_i^{s}}\right)^{t}\right)
=\prod_{n\ge1}\left(\sum_{t\ge0}g_\alpha(t)\left(\frac{x}{n^{s}}\right)^{t}\right).
\]
Here, each term in the expansion of the last product corresponds to a finite set $n_1<\cdots<n_\ell$ of indices at which a factor with $t\ge1$ is taken. This is the assertion.
\end{proof}

\begin{cor}\label{cor:s2}
For $\alpha\ne0$ we have
\[
\sum_{r\ge0}\frac{1}{r!}\zeta_{(r)}^J\left(\{2\}^r;\alpha\right)x^r
=\left(\frac{\sin \pi\sqrt{\alpha x}}{\pi\sqrt{\alpha x}}\right)^{-\frac1\alpha}.
\]
In particular $\frac{1}{r!}\zeta_{(r)}^J\left(\{2\}^r;\alpha\right)$ is $\pi^{2r}$ times a polynomial in $\alpha$ with rational coefficients.
\end{cor}

\begin{proof}
Put $z=\sqrt{\alpha x}$ in the Euler product $\prod_{n\ge1}\left(1-\frac{z^2}{n^2}\right)=\frac{\sin\pi z}{\pi z}$ and apply Proposition \ref{prop:genfun}.
\end{proof}

\begin{exam}
The first values are
\begin{align*}
\frac{1}{1!}\zeta_{(1)}^J\left(\{2\}^1;\alpha\right)&=\frac{\pi^2}{6}, \\
\frac{1}{2!}\zeta_{(2)}^J\left(\{2\}^2;\alpha\right)&=\frac{(2\alpha+5)\pi^4}{360},\\
\frac{1}{3!}\zeta_{(3)}^J\left(\{2\}^3;\alpha\right)&=\frac{(16\alpha^2+42\alpha+35)\pi^6}{45360}, \\
\frac{1}{4!}\zeta_{(4)}^J\left(\{2\}^4;\alpha\right)&=\frac{(4\alpha+5)(36\alpha^2+56\alpha+35)\pi^8}{5443200}.
\end{align*}
At $\alpha=-1$ these values are $\zeta(\{2\}^r)=\frac{\pi^{2r}}{(2r+1)!}$, and at $\alpha=1$ they are $\zetas(\{2\}^r)=2\left(1-2^{1-2r}\right)\zeta(2r)$, namely
\[\frac{\pi^2}{6},\quad\frac{7\pi^4}{360},\quad\frac{31\pi^6}{15120},\quad\frac{127\pi^8}{604800}.\]
\end{exam}

\subsection{Values at $\{2\}^a,3,\{2\}^b$}

In this subsection, we show a reflection formula for the values at $\{2\}^a,3,\{2\}^b$, which reduces to a relation of Zagier at $\alpha=1$. Following Zagier \cite{z12}, we insert one $3$ among the $2$'s. For integers $a,b\ge0$ we write $\{2\}^a,3,\{2\}^b$ for the tableau of shape $(a+b+1)$ whose entries are $2,\ldots,2,3,2,\ldots,2$ with $a$ entries $2$ before the $3$, and put
\[V_{a,b}(\alpha)=\frac{1}{(a+b+1)!}\zeta_{(a+b+1)}^J\left(\{2\}^a,3,\{2\}^b;\alpha\right).\]
By Lemma \ref{lem:1r}, $V_{a,b}(\alpha)$ is a polynomial in $\alpha$ whose coefficients are $\mathbb{Q}$-linear combinations of multiple zeta values of weight $w=2a+2b+3$, and $V_{a,b}(-1)=\zeta(\{2\}^a,3,\{2\}^b)$ by Corollary \ref{cor:interp}.
For odd $w$ we denote by $\mathcal{M}_w$ the $\mathbb{Q}$-vector space spanned by the multiple zeta values of weight $w$, and by
\[\mathcal{Z}_w=\mathbb{Q}\zeta(w)+\mathbb{Q}\pi^2\zeta(w-2)+\cdots+\mathbb{Q}\pi^{w-3}\zeta(3)\subset\mathcal{M}_w.\]
Zagier \cite{z12} gave explicit evaluations of $\zeta(\{2\}^a,3,\{2\}^b)$ and $\zetas(\{2\}^a,3,\{2\}^b)$ \cite[Theorems~1 and 4]{z12}, which show that both of them lie in $\mathcal{Z}_{2a+2b+3}$ for all $a,b\ge0$. We refer to this as Zagier's theorem.

For $\alpha\ne0$ and $m\ge1$ put $Q_\alpha(x,m)=\left(1-\frac{\alpha x^2}{m^2}\right)^{-\frac1\alpha}$ and $P_\alpha(x)=\prod_{m\ge1}Q_\alpha(x,m)$. By Corollary \ref{cor:s2},
\[P_\alpha(x)=\sum_{r\ge0}c_r(\alpha)x^{2r},\qquad c_r(\alpha)=\frac{1}{r!}\zeta_{(r)}^J\left(\{2\}^r;\alpha\right)\in\mathbb{Q}[\alpha]\pi^{2r}.\]

\begin{prop}\label{prop:zaggen}
We have
\begin{align*}
F_\alpha(x,y)&:=\sum_{a,b\ge0}V_{a,b}(\alpha)x^{2a}y^{2b}\\
&\phantom{:}=\sum_{n\ge1}\left(\prod_{m<n}Q_\alpha(x,m)\right)\frac{Q_\alpha(x,n)-Q_\alpha(y,n)}{n\,(x^2-y^2)}\left(\prod_{m>n}Q_\alpha(y,m)\right).
\end{align*}
\end{prop}

\begin{proof}
As in the proof of Proposition \ref{prop:genfun}, we expand $V_{a,b}(\alpha)$ by Lemma \ref{lem:1r} and write each multiple zeta value as a sum over $n_1<\cdots<n_\ell$.
We attach $x^2$ to each variable before the $3$ and $y^2$ to each variable after it. A block of size $t$ lying before the block containing the $3$ contributes $g_\alpha(t)x^{2t}n_i^{-2t}$, and one lying after it contributes $g_\alpha(t)y^{2t}n_i^{-2t}$; summing over $t\ge0$ gives $Q_\alpha(x,n_i)$ and $Q_\alpha(y,n_i)$ respectively.
If the block containing the $3$ has size $t$ and index $n$, and the $3$ is its $p$-th entry, it contributes $g_\alpha(t)x^{2(p-1)}y^{2(t-p)}n^{-(2t+1)}$. Since $\sum_{p=1}^{t}x^{2(p-1)}y^{2(t-p)}=\frac{x^{2t}-y^{2t}}{x^2-y^2}$, summing over $p$ and $t\ge1$ gives $\frac{Q_\alpha(x,n)-Q_\alpha(y,n)}{n\,(x^2-y^2)}$.
\end{proof}

\begin{theorem}\label{thm:refl}
We have $F_\alpha(y,x)=P_\alpha(x)P_\alpha(y)F_{-\alpha}(ix,iy)$, that is,
\[
V_{b,a}(\alpha)=\sum_{a'=0}^{a}\sum_{b'=0}^{b}(-1)^{a'+b'}c_{a-a'}(\alpha)c_{b-b'}(\alpha)V_{a',b'}(-\alpha)
\qquad(a,b\ge0).
\]
\end{theorem}

\begin{proof}
Since $Q_{-\alpha}(ix,m)=\left(1-\frac{\alpha x^2}{m^2}\right)^{\frac1\alpha}=Q_\alpha(x,m)^{-1}$, we write $Q_x(m)=Q_\alpha(x,m)$ and $Q_y(m)=Q_\alpha(y,m)$ and obtain from Proposition \ref{prop:zaggen}
\[
F_{-\alpha}(ix,iy)=\sum_{n\ge1}\left(\prod_{m<n}Q_x(m)^{-1}\right)\frac{Q_x(n)^{-1}-Q_y(n)^{-1}}{-n\,(x^2-y^2)}\left(\prod_{m>n}Q_y(m)^{-1}\right).
\]
Here $\frac{Q_x(n)^{-1}-Q_y(n)^{-1}}{-(x^2-y^2)}=\frac{Q_x(n)-Q_y(n)}{(x^2-y^2)Q_x(n)Q_y(n)}$, and
\[\prod_{m\le n}Q_x(m)^{-1}=P_\alpha(x)^{-1}\prod_{m>n}Q_x(m),\qquad
\prod_{m\ge n}Q_y(m)^{-1}=P_\alpha(y)^{-1}\prod_{m<n}Q_y(m).\]
Hence
\[
P_\alpha(x)P_\alpha(y)F_{-\alpha}(ix,iy)=\sum_{n\ge1}\left(\prod_{m<n}Q_y(m)\right)\frac{Q_y(n)-Q_x(n)}{n\,(y^2-x^2)}\left(\prod_{m>n}Q_x(m)\right)=F_\alpha(y,x).
\]
Comparing the coefficients of $x^{2a}y^{2b}$ gives the second form.
\end{proof}

We put $\alpha=1$ in Theorem \ref{thm:refl}. Since $V_{a,b}(1)=\zetas(\{2\}^a,3,\{2\}^b)$, $V_{a,b}(-1)=\zeta(\{2\}^a,3,\{2\}^b)$ and $c_r(1)=\zetas(\{2\}^r)$ by Corollary \ref{cor:interp}, we obtain for $a,b\ge0$
\[\zetas(\{2\}^b,3,\{2\}^a)=\sum_{a'=0}^{a}\sum_{b'=0}^{b}(-1)^{a'+b'}\zetas(\{2\}^{a-a'})\zetas(\{2\}^{b-b'})\zeta(\{2\}^{a'},3,\{2\}^{b'}).\]
This is the relation between the generating functions of these values in \cite[Section~7]{z12}, from which Zagier deduced the evaluation of $\zetas(\{2\}^a,3,\{2\}^b)$ from that of $\zeta(\{2\}^a,3,\{2\}^b)$. Thus Theorem \ref{thm:refl} interpolates Zagier's relation by the parameter $\alpha$.

\begin{cor}\label{cor:refl}
\begin{enumerate}
\item Let $\overline{V_{a,b}}(\alpha)\in(\mathcal{M}_w/\mathcal{Z}_w)[\alpha]$ denote the image of
$V_{a,b}(\alpha)$. Then $\overline{V_{a,b}}(\alpha)$ is divisible by $1-\alpha^2$.
\item Modulo $\zeta(2)\mathcal{M}_{w-2}$ we have
$V_{b,a}(\alpha)\equiv(-1)^{a+b}V_{a,b}(-\alpha)$.
\end{enumerate}
\end{cor}

\begin{proof}
(1) By Corollary \ref{cor:interp}, $V_{a,b}(-1)=\zeta(\{2\}^a,3,\{2\}^b)$ and $V_{a,b}(1)=\zetas(\{2\}^a,3,\{2\}^b)$, which lie in $\mathcal{Z}_w$ by Zagier's theorem. Therefore, $\overline{V_{a,b}}(\alpha)$ vanishes at $\alpha=-1$ and at $\alpha=1$.
(2) Since $c_r(\alpha)\in\mathbb{Q}[\alpha]\pi^{2r}$, all the terms with $(a',b')\ne(a,b)$ in Theorem \ref{thm:refl} lie in $\zeta(2)\mathcal{M}_{w-2}$, while the remaining term is $(-1)^{a+b}V_{a,b}(-\alpha)$.
\end{proof}

For $w=11$, that is, for all $(a,b)$ with $a+b=4$, we found numerically that
\[V_{a,b}(\alpha)\in\mathcal{Z}_{11}[\alpha]+\frac{(-1)^a}{3}\binom{4}{a}(1-\alpha^2)\zeta(3,3,5).\]
This leads to the following problem.

\begin{prob}
Determine the quotient $\overline{V_{a,b}}(\alpha)/(1-\alpha^2)$ for $w\ge13$. In particular, is it always independent of $\alpha$?
\end{prob}

\appendix
\section{The shape $(r,1)$ and a telescoping identity}\label{app:r1}

In this appendix, we give a direct proof of Theorem \ref{thm:general} for $\lambda=(r,1)$ with $r\ge2$. In this case, we have $\lambda'=(2,1^{r-1})$ and $n=r+1$, and we can write both sides of Theorem \ref{thm:general} explicitly. Then the duality follows from a telescoping identity, which shows the cancellation behind Lemma \ref{lem:local} in the simplest nontrivial case.

For an integer $v\ge0$ we put
\[\varphi(v)=\frac{v}{1+\alpha(v-1)},\qquad
\gamma(v)=\frac{(1+\alpha r)\bigl(2+\alpha(v-1)\bigr)}{1+\alpha v},\]
so that $\varphi(0)=0$. For a composition $\boldsymbol{t}=(t_1,\ldots,t_\ell)$ of $r$ we write $\tau_k=t_1+\cdots+t_k$ $(0\le k\le\ell)$ and $\sigma_i=s_{1(\tau_{i-1}+1)}+\cdots+s_{1\tau_i}$ $(1\le i\le\ell)$, and we consider the indices
\begin{align*}
\pmb u^{\rm I}_k(\boldsymbol{t})&=(\sigma_1,\ldots,\sigma_{k-1},\ \sigma_k+s_{21},\ \sigma_{k+1},\ldots,\sigma_\ell)&&(2\le k\le\ell),\\
\pmb u^{\rm II}_k(\boldsymbol{t})&=(\sigma_1,\ldots,\sigma_{k},\ s_{21},\ \sigma_{k+1},\ldots,\sigma_\ell)&&(1\le k\le\ell).
\end{align*}
Here, $g_\alpha$ is the function defined in Section \ref{sec:omega}.

\begin{prop}\label{prop:r1explicit}
\begin{enumerate}
\item We have
\begin{align*}
\zeta_{(r,1)}^J(\pmb s;\alpha)
=(r-1)!\sum_{\boldsymbol{t}}\ \prod_{i=1}^{\ell}g_\alpha(t_i)
\Biggl(&\sum_{k=2}^{\ell}\gamma(\tau_k)\varphi(\tau_{k-1})\zeta\bigl(\pmb u^{\rm I}_k(\boldsymbol{t})\bigr)\\
&+\sum_{k=1}^{\ell}\gamma(\tau_k)\varphi(\tau_{k})\zeta\bigl(\pmb u^{\rm II}_k(\boldsymbol{t})\bigr)\Biggr),
\end{align*}
where $\boldsymbol{t}=(t_1,\ldots,t_\ell)$ runs over all compositions of $r$.
\item We have
\begin{align*}
\zeta_{(2,1^{r-1})}^J\left(\pmb s';\frac1\alpha\right)
&=(r-1)!\sum_{v=1}^{r}\gamma(v)\varphi(v)\zeta(s_{11},\ldots,s_{1v},s_{21},s_{1(v+1)},\ldots,s_{1r})\\
&\qquad+(r-1)!\,\frac{1+\alpha r}{\alpha}\sum_{v=1}^{r}\zeta(s_{11},\ldots,s_{1(v-1)},s_{1v}+s_{21},s_{1(v+1)},\ldots,s_{1r}).
\end{align*}
\end{enumerate}
\end{prop}

\begin{proof}
We first consider (1). In a chain from $\emptyset$ to $(r,1)$, the first row grows by blocks of sizes $t_1,\ldots,t_\ell$. Since the first block contains $(1,1)$, the box $(2,1)$ cannot be added together with the first block. Therefore, the box $(2,1)$ is added either together with the $k$-th block for some $k\ge2$, or alone right after the $k$-th block.
The former gives the index $\pmb u^{\rm I}_k(\boldsymbol{t})$ and the latter gives $\pmb u^{\rm II}_k(\boldsymbol{t})$. As in the proof of Lemma \ref{lem:1r}, for $m,t\ge1$, we have
\begin{align*}
\psi_{(m+t)/(m)}(\alpha)&=\binom{m+t}{t}\prod_{j=0}^{t-1}(1+\alpha j),\\
\psi_{(m+t,1)/(m,1)}(\alpha)&=\binom{m+t}{t}\prod_{j=0}^{t-1}(1+\alpha j)\cdot\frac{1+\alpha(m+t)}{1+\alpha m}\cdot\frac{m}{m+t},\\
\psi_{(m+t,1)/(m)}(\alpha)&=\binom{m+t}{t}\prod_{j=0}^{t-1}(1+\alpha j)\cdot\frac{2+\alpha(m+t-1)}{1+\alpha(m-1)}\cdot\frac{m}{m+t},\\
\psi_{(m,1)/(m)}(\alpha)&=\frac{2+\alpha(m-1)}{1+\alpha(m-1)},
\end{align*}
and the first formula also holds for $m=0$. Taking the product over the steps of a chain, the binomial coefficients and the products over $j$ give $r!\prod_{i}g_\alpha(t_i)$, and the steps after the one in which $(2,1)$ is added give
\[\prod_{i>k}\frac{1+\alpha \tau_i}{1+\alpha \tau_{i-1}}\cdot\frac{\tau_{i-1}}{\tau_i}=\frac{1+\alpha r}{1+\alpha \tau_k}\cdot\frac{\tau_k}{r}.\]
Moreover, the step in which $(2,1)$ is added gives the additional factor $\frac{2+\alpha(\tau_k-1)}{1+\alpha(\tau_{k-1}-1)}\cdot\frac{\tau_{k-1}}{\tau_k}$ in the former case and $\frac{2+\alpha(\tau_k-1)}{1+\alpha(\tau_{k}-1)}$ in the latter case. Therefore, the weights are
\[(r-1)!\prod_{i=1}^{\ell}g_\alpha(t_i)\,\gamma(\tau_k)\varphi(\tau_{k-1})\qquad\text{and}\qquad
(r-1)!\prod_{i=1}^{\ell}g_\alpha(t_i)\,\gamma(\tau_k)\varphi(\tau_{k}),\]
respectively.

We next consider (2). In a chain from $\emptyset$ to $(2,1^{r-1})$, the boxes of the first column are added one at a time. The box $(1,2)$, which carries $s_{21}$, is added either alone right after the $v$-th box of the first column or together with it $(1\le v\le r)$.
For a parameter $\beta$ and $m,v\ge1$, we have
\begin{align*}
\psi_{(1^m)/(1^{m-1})}(\beta)&=m, &
\psi_{(2,1^{m})/(2,1^{m-1})}(\beta)&=\frac{m(\beta+m+1)}{\beta+m},\\
\psi_{(2,1^{v-1})/(1^{v})}(\beta)&=\frac{2\beta+v-1}{\beta+v-1}, &
\psi_{(2,1^{v-1})/(1^{v-1})}(\beta)&=\beta+v.
\end{align*}
Therefore, the weights of the two kinds of chains are
\[(r-1)!\,\frac{v(2\beta+v-1)(\beta+r)}{(\beta+v-1)(\beta+v)}\qquad\text{and}\qquad(r-1)!\,(\beta+r),\]
respectively. Putting $\beta=\frac1\alpha$, we obtain the assertion.
\end{proof}

\begin{lemma}\label{lem:telescope}
For integers $a,t\ge0$, we have
\begin{equation}\label{eqn:telescope}
\sum_{v=a}^{a+t}\frac{v\bigl(2+\alpha(v-1)\bigr)}{(1+\alpha v)\bigl(1+\alpha(v-1)\bigr)}
=\frac{(a+t)(a+t+1)}{1+\alpha(a+t)}-\frac{a(a-1)}{1+\alpha(a-1)}.
\end{equation}
Consequently, we have
\begin{equation}\label{eqn:telescope2}
\frac{1+\alpha t}{t+1}\sum_{v=a}^{a+t}\gamma(v)\varphi(v)-t(1+\alpha r)=\gamma(a+t)\varphi(a).
\end{equation}
\end{lemma}

\begin{proof}
Let $\theta(v)=\frac{v(v-1)}{1+\alpha(v-1)}$. Then we have
\[\theta(v+1)-\theta(v)=\frac{v(v+1)}{1+\alpha v}-\frac{v(v-1)}{1+\alpha(v-1)}
=\frac{v\bigl(2+\alpha(v-1)\bigr)}{(1+\alpha v)\bigl(1+\alpha(v-1)\bigr)}.\]
Therefore, the left-hand side of \eqref{eqn:telescope} is equal to $\theta(a+t+1)-\theta(a)$, and we obtain \eqref{eqn:telescope}. Since $\gamma(v)\varphi(v)$ is $1+\alpha r$ times the summand of \eqref{eqn:telescope}, the identity \eqref{eqn:telescope2} is equivalent to
\[\frac{1+\alpha t}{t+1}\left(\frac{(a+t)(a+t+1)}{1+\alpha(a+t)}-\frac{a(a-1)}{1+\alpha(a-1)}\right)-t
=\frac{a\bigl(2+\alpha(a+t-1)\bigr)}{\bigl(1+\alpha(a+t)\bigr)\bigl(1+\alpha(a-1)\bigr)},\]
which follows from a direct calculation.
\end{proof}

\begin{proof}[Direct proof of Theorem \ref{thm:general} for $\lambda=(r,1)$]
We apply $\omega_{-\frac1\alpha}$ to Proposition \ref{prop:r1explicit}(2) and compare the coefficients with those of Proposition \ref{prop:r1explicit}(1). Let $\boldsymbol{t}=(t_1,\ldots,t_\ell)$ be a composition of $r$.
Since $\binom{-1/\alpha}{t}=(-\alpha)^{-t}g_\alpha(t)$, we have
\[\prod_{i=1}^{\ell}\binom{-1/\alpha}{t_i}=(-\alpha)^{-r}\prod_{i=1}^{\ell}g_\alpha(t_i).\]

First, the index $\pmb u^{\rm II}_k(\boldsymbol{t})$ is obtained by merging only from the term with $v=\tau_k$ of the first sum in Proposition \ref{prop:r1explicit}(2), with the coefficient $\binom{-1/\alpha}{1}\prod_i\binom{-1/\alpha}{t_i}$. Therefore, its coefficient after applying $\omega_{-\frac1\alpha}$ is
\[(-\alpha)^{-r-1}(r-1)!\prod_{i=1}^{\ell}g_\alpha(t_i)\,\gamma(\tau_k)\varphi(\tau_k).\]

Next, the index $\pmb u^{\rm I}_k(\boldsymbol{t})$ is obtained from the terms with $\tau_{k-1}\le v\le \tau_k$ of the first sum, where $t_k+1$ entries are merged into the $k$-th entry, and from the terms with $\tau_{k-1}<v\le \tau_k$ of the second sum, where $t_k$ entries are merged into the $k$-th entry.
Here, the term with $v=0$ does not occur, but we may include it in the first range because $\varphi(0)=0$. Since $\binom{-1/\alpha}{t_k+1}=-\frac{1+\alpha t_k}{\alpha(t_k+1)}\binom{-1/\alpha}{t_k}$, the coefficient of $\zeta\bigl(\pmb u^{\rm I}_k(\boldsymbol{t})\bigr)$ is
\begin{align*}
&(r-1)!\prod_{i=1}^{\ell}\binom{-1/\alpha}{t_i}\cdot\left(-\frac{1}{\alpha}\right)
\left(\frac{1+\alpha t_k}{t_k+1}\sum_{v=\tau_{k-1}}^{\tau_k}\gamma(v)\varphi(v)-t_k(1+\alpha r)\right)\\
&=(-\alpha)^{-r-1}(r-1)!\prod_{i=1}^{\ell}g_\alpha(t_i)\,\gamma(\tau_k)\varphi(\tau_{k-1})
\end{align*}
by \eqref{eqn:telescope2} with $a=\tau_{k-1}$ and $t=t_k$.

In both cases, the coefficient is $(-\alpha)^{-r-1}$ times that in Proposition \ref{prop:r1explicit}(1). Moreover, no other index appears with a nonzero coefficient, because $\gamma(\tau_1)\varphi(\tau_0)=0$. Multiplying by $(-1)^{r+1}$, we have
\[(-1)^{r+1}\omega_{-\frac{1}{\alpha}}\left(\zeta_{(2,1^{r-1})}^J\left(\pmb s ';\frac{1}{\alpha}\right)\right)=\frac{1}{\alpha^{r+1}}\zeta_{(r,1)}^J\left(\pmb s;\alpha\right).\qedhere\]
\end{proof}

\section*{Acknowledgement}
This work was supported by a Grant-in-Aid for Scientific Research (C) (Grant Number: JP22K03274) and Grant-in-Aid for Early-Career Scientists (Grant Number: JP22K13900).

\end{document}